\pdfoutput=1

\documentclass[fontsize=10pt,oneside,
bibtotoc,liststotoc,idxtotoc,
cleardoublepage=plain,paper=A4,
tablecaptionbelow,onelinecaption,
abstracton,
largeheadings,pagesize]
{scrartcl}

\usepackage{lmodern}
\usepackage[utf8]{inputenc}   
\usepackage[LY1,T1]{fontenc}
\usepackage{fixltx2e} 
\usepackage[english]{babel}  
\usepackage[protrusion=alltext,expansion=basicmath,
            kerning, spacing, 
             verbose,final,factor=800
            ]{microtype}

\usepackage[ntheorem]{empheq}  

\usepackage[numbers,square,comma,sort]{natbib} 

\usepackage{varioref} 
\usepackage{lastpage}

\usepackage{booktabs} 
\usepackage{tabularx} 

\usepackage{xspace}           
\usepackage[x11names,svgnames]{xcolor} 
\usepackage[babel]{csquotes}  

\usepackage{ragged2e} 

\usepackage[perpage,stable,multiple]{footmisc}
\usepackage{enumitem} 

\usepackage{float} 

\usepackage{fancyvrb} 

\usepackage{graphicx} 
\usepackage{pdfpages} 
\usepackage{tikz}
\usepackage[colorlinks,linkcolor=MidnightBlue,anchorcolor=blue,citecolor=green,
            breaklinks,pdfborder={0 0 0},bookmarks=true]{hyperref}
\usepackage[xindy,sanitize={name=false,description=false}]{glossaries}

\PackageInfo{}{############## End of package loading #####################}

\usetikzlibrary{arrows,matrix,fit}

\DeclareUnicodeCharacter{00A0}{~}
\DeclareUnicodeCharacter{00AB}{\og\ignorespaces} 
\DeclareUnicodeCharacter{00BB}{\fg}   

\setdescription{noitemsep}

\tikzset{
  information text/.style={rounded corners,fill=blue!10,inner sep=2ex},
  myarrow/.style={->,shorten >=1pt},
  mydiag/.style={auto, text height=1.5ex, text depth=0.25ex},
  mymatrix/.style={row sep=1.5cm, column sep=1.5cm, 
    matrix of math nodes, ampersand replacement=\& },
  mymatrix0/.style={matrix of math nodes, ampersand replacement=\& }
}

\PackageWarning{}{############## End of package settings #####################}

\usepackage{typearea}
\usepackage[headsepline,footsepline,plainheadsepline,plainfootsepline,
            automark,autooneside]
           {scrpage2}
\usepackage[amsmath,hyperref,thmmarks]{ntheorem}   

\theoremstyle{plain}
\theoremstyle{break}
\theoremseparator{:}
\theoremsymbol{}
\def\theoremheaderindent{\indent}

\theorembodyfont{\itshape}
\theoremheaderfont{\upshape\bfseries}
\newtheorem{theorem}{\theoremheaderindent Theorem}[section]
\newtheorem{proposition}[theorem]{\theoremheaderindent Proposition}
\newtheorem{corollary}[theorem]{\theoremheaderindent Corollary}
\newtheorem{lemma}[theorem]{\theoremheaderindent Lemma}

\theorembodyfont{\upshape}
\newtheorem{definition}[theorem]{\theoremheaderindent Definition}

\newtheorem{example}[theorem]{\theoremheaderindent Example}

\theoremsymbol{\ensuremath{\diamondsuit}}
\newtheorem{remark}[theorem]{\theoremheaderindent Remark}
\newtheorem{algorithm}[theorem]{\theoremheaderindent Algorithm}
\newtheorem{complexity}[theorem]{\theoremheaderindent Complexity Analysis}
\newtheorem{algocomplexity}[theorem]{\theoremheaderindent Correction and Complexity Analysis}

\theoremstyle{nonumberplain}
\theoremheaderfont{\itshape}
\theoremsymbol{\ensuremath{\blacksquare}}
\newtheorem{proof}{Proof}

\qedsymbol{\ensuremath{\diamondsuit}}

\newtagform{color}{(\colorlet{default}{.}\color{MidnightBlue}}{\color{default})}
\usetagform{color}
\setkomafont{sectioning}{\sffamily\color{MidnightBlue}}
\setkomafont{title}{\rmfamily\color{MidnightBlue}}

\title{Computing Isogenies Between Abelian Varieties}
\author{David Lubicz$^{1,2}$, Damien Robert$^{3}$}
\date{\normalsize}
\publishers{
\footnotesize
$^{1}$ C\'ELAR,\\
BP 7419,
F-35174 Bruz\\
$^{2}$ IRMAR, Universt\'e de Rennes 1,\\
Campus de Beaulieu,
F-35042 Rennes\\
$^{3}$ LORIA, Campus Scientifique, BP 239,\\
F-54506 Vand\oe{}uvre-l\`es-Nancy\\
}
\usepackage{mathpazo}
\usepackage{mathrsfs}
\usepackage{amssymb}

\DeclareMathOperator{\chaineadd}{\mathtt{chaine\_add}}
\DeclareMathOperator{\chaineaddmult}{\mathtt{chaine\_multadd}}
\DeclareMathOperator{\chainemultadd}{\mathtt{\chaineaddmult}}
\DeclareMathOperator{\chainemult}{\mathtt{chaine\_mult}}

\renewcommand{\theta}{\vartheta}
\renewcommand{\tilde}{\widetilde}

\newcommand{\Inv}{\mathfrak{I}}

\newcommand{\Scal}{\mathcal{S}}
\newcommand{\Sfrak}{\mathfrak{S}}
\newcommand{\pA}{\ensuremath{p_{A_k}}}
\newcommand{\pB}{\ensuremath{p_{B_k}}}
\newcommand{\Bell}{\ensuremath{\tilde{B_k}'}}
\newcommand{\elliso}{\ensuremath{[\ell]}}
\newcommand{\ellisotilde}{\ensuremath{\tilde{[\ell]}}}
\newcommand{\Aff}{\ensuremath{\mathbb{A}}}
\newcommand{\Atilde}{\ensuremath{\tilde{A}_k}}
\newcommand{\Btilde}{\ensuremath{\tilde{B}_k}}
\newcommand{\Ptilde}{\ensuremath{\tilde{P}}}
\newcommand{\Qtilde}{\ensuremath{\tilde{Q}}}
\newcommand{\Rtilde}{\ensuremath{\tilde{R}}}
\newcommand{\Ktilde}{\ensuremath{\tilde{K}}}
\newcommand{\thetatilde}{\ensuremath{\tilde{\theta}}}
\newcommand{\utilde}{\ensuremath{\tilde{u}}}
\newcommand{\xtilde}{\ensuremath{\tilde{x}}}
\newcommand{\ytilde}{\ensuremath{\tilde{y}}}

\newcommand{\xytilde}{\ensuremath{\tilde{x+y}}}
\newcommand{\xymtilde}{\ensuremath{\tilde{x-y}}}
\newcommand{\zeroA}{\ensuremath{\widetilde{0}_{A_k}}}
\newcommand{\zeroE}{\ensuremath{\widetilde{0}_{E_k}}}
\newcommand{\zeroB}{\ensuremath{\widetilde{0}_{B_k}}}
\newcommand{\zeroBl}{\ensuremath{\widetilde{0}_{\Bell}}}

\newcommand{\rhopol}{\rho_{\pol}}
\newcommand{\rhopolz}{\rho_{\pol_0}}
\newcommand{\rhopola}{\tilde{\rho}_{\pol}}
\newcommand{\rhopolaz}{\tilde{\rho}_{\pol_0}}

\newcommand{\overln}{\ensuremath{\overline{\ell n}}}
\newcommand{\overn}{\ensuremath{\overline{n}}}
\newcommand{\overl}{\ensuremath{\overline{\ell}}}
\newcommand{\overll}{\ensuremath{\overline{\ell^2 n}}}
\newcommand{\overtwo}{\ensuremath{\overline{2}}}

\newcommand{\deltab}{\delta'}
\newcommand{\deltao}{\delta_0}
\newcommand{\pitilde}{\tilde{\pi}}
\newcommand{\pidual}{\hat{\pi}}
\newcommand{\Thetastruct}[1]{\ensuremath{\Theta_{#1}}}

\newcommand{\ThetaA}{\ensuremath{\Theta_{A_k}}}
\newcommand{\ThetaB}{\ensuremath{\Theta_{B_k}}}
\newcommand{\ThetabarA}{\overline{\Theta}_{A_k}}

\newcommand{\Hstruct}[1]{\ensuremath{\mathcal{H}(#1)}}
\newcommand{\Zstruct}[1]{\ensuremath{Z({#1})}}
\newcommand{\dZstruct}[1]{\hat{Z}(#1)}
\newcommand{\Kstruct}[1]{\ensuremath{K({#1})}}
\newcommand{\Zln}{\Zstruct{\overln}}
\newcommand{\Zn}{\Zstruct{\overn}}
\newcommand{\Zl}{\Zstruct{\overl}}
\newcommand{\Zll}{\Zstruct{\overll}}
\newcommand{\Zdelta}{\Zstruct{\delta}}
\newcommand{\Zdeltab}{\Zstruct{\delta'}}
\newcommand{\Zdeltao}{\Zstruct{\delta_0}}
\newcommand{\dZln}{\dZstruct{\overln}}
\newcommand{\dZn}{\dZstruct{\overn}}
\newcommand{\dZl}{\dZstruct{\overl}}
\newcommand{\dZll}{\dZstruct{\overll}}
\newcommand{\dZdelta}{\dZstruct{\delta}}

\newcommand{\Hln}{\Hstruct{\overln}}

\newcommand{\Hll}{\Hstruct{\overll}}
\newcommand{\Hdelta}{\Hstruct{\delta}}

\newcommand{\Kln}{\Kstruct{\overln}}

\newcommand{\Kdelta}{\Kstruct{\delta}}

\newcommand{\Thetall}{\Thetastruct{B_k, \bpol}}

\newcommand{\Mstruct}[1]{\mathcal{M}_{#1}}
\newcommand{\Mln}{\Mstruct{\overln}}
\newcommand{\Mn}{\Mstruct{\overn}}

\newcommand{\Mdelta}{\Mstruct{\delta}}

\newcommand{\Ztwo}{\Zstruct{\overtwo}}
\newcommand{\dZtwo}{\dZstruct{\overtwo}}
\newcommand{\Az}{A^0}

\newcommand{\Kpol}{K(\pol)}
\newcommand{\KpolO}{K(\pol_0)}
\newcommand{\Gpol}{G(\pol)}
\newcommand{\GpolO}{G(\pol_0)}
\newcommand{\Kpola}{K_1(\pol)}
\newcommand{\Kpolb}{K_2(\pol)}
\newcommand{\epol}{e_{\pol}}
\newcommand{\edelta}{e_{c,\delta}}

\newcommand{\pullb}[1]{{#1}^{*}}

\newcommand{\proj}{\mathbb P}

\newcommand{\iso}{\stackrel{\sim}{\rightarrow}}
\newcommand{\pol}{\mathscr{L}}
\newcommand{\Mpol}{\mathscr{M}}

\newcommand{\ppol}{\mathscr{L}_0}
\newcommand{\bpol}{\mathscr{M}_0}

\newcommand{\mbb}{\mathbb}

\newcommand{\N}{\ensuremath{\mbb{N}}}
\newcommand{\Nstar}{\ensuremath{\mbb{N}^{\ast}}}
\newcommand{\Z}{\ensuremath{\mbb{Z}}}

\newcommand{\F}{\ensuremath{\mbb{F}}}
\newcommand{\Pvar}{\ensuremath{\mbb{P}}}
\newcommand{\Avar}{\ensuremath{\mbb{A}}}

\newcommand{\Ze}{\mathcal{Z}}

\renewcommand{\geq}{\ensuremath{\geqslant}}

\renewcommand*{\phi}{\varphi}
\renewcommand*{\epsilon}{\varepsilon}

\newcommand{\sub}{\ensuremath{\subset}}

\newcommand{\union}{\ensuremath{\bigcup}}

\newcommand{\prodtens}{\otimes}
\newcommand{\tens}{\otimes}

\DeclareMathOperator{\Aut}{Aut}

\newcommand{\kbar}{\overline{k}}

\glossarystyle{long3colheader}
\makeglossaries

\let\myglossaryentry\newglossaryentry

\myglossaryentry{glo@Zn}{name={$\Zn$},description={$\Z^g/n\Z^g$},sort={Zn},sort={00}}
\myglossaryentry{glo@Mn}{name={$\Mn$},description={The moduli space of theta
null points of level $n$.},sort={00a}}
\myglossaryentry{glo@Ak}{name={$A_k$},description={$(A_k, \pol, \Theta_{A_k})$ is a polarized abelian variety with a theta structure of level $\ell n$.}, sort={00b}}
\myglossaryentry{glo@Bk}{name={$B_k$},description={$(B_k, \pol_0, \Theta_{B_k})$ is an abelian variety $\ell$-isogenous to $A_k$ with a theta structure of level $n$.},sort={01}}
\myglossaryentry{glo@thetai}{name={$\theta_i$},description={$(\theta_i)_{i \in \Zln}$ are the canonical projective coordinates on $A_k$ given by the theta structure.},sort={02}}
\myglossaryentry{glo@zeroA}{name={$0_{A_k}$},description={The theta null $0_{A_k}=\theta_i(0_{A_k})_{i \in \Zln}$.},sort={03}}
\myglossaryentry{glo@zeroB}{name={$0_{B_k}$},description={The theta null $0_{B_k}=\theta_i(0_{B_k})_{i \in \Zn}$.},sort={04}}

\myglossaryentry{glo@Gpol}{name={$G(\pol)$},description={The Theta group of $(A_k, \pol)$},sort={06}}
\myglossaryentry{glo@Kpol}{name={$\Kpol$ },description={$\Kpol=\Kpola \oplus \Kpolb$ is the decomposition of the kernel of the polarization $\pol$ induced by the Theta structure $\Theta_A$},sort={07}}
\myglossaryentry{glo@Hdelta}{name={$H(\delta)$},description={The Heisenberg group of type $\delta$.},sort={08}}
\myglossaryentry{glo@sKpola}{name={$s_{\Kpola}$},description={The natural section $\Kpola \to \Gpol$ induced by the Theta structure.},sort={09}}
\myglossaryentry{glo@rhopola}{name={$\rhopola^*$},description={The affine action of $G(\pol)$ on $\Atilde$.},sort={090}}
\myglossaryentry{glo@rhopol}{name={$\rhopol^*$},description={The projective action of $\Kpol$ on $A_k$.},sort={091}}

\myglossaryentry{glo@Atilde}{name={$\Atilde$},description={The affine cone of
$A_k$.},sort={10}}
\myglossaryentry{glo@Btilde}{name={$\Btilde$},description={The affine cone of
$B_k$.},sort={11}}
\myglossaryentry{glo@thetatildei}{name={$\thetatilde_i$},description={$(\thetatilde_i)_{i \in \Zln}$ are the affine coordinates on $\Atilde$.},sort={12}}
\myglossaryentry{glo@zerotildeB}{name={$\zeroB$},description={An affine lift of $0_{B_k}$.},sort={13}}
\myglossaryentry{glo@zerotildeA}{name={$\zeroA$},description={The affine lift of $0_{A_k}$ such that $\pitilde(\zeroA)=\zeroB$.},sort={14}}

\myglossaryentry{glo@pi}{name={$\pi$},description={The $\ell$-isogeny $\pi: A_k \to B_k$.},sort={20}}
\myglossaryentry{glo@pitilde}{name={$\pitilde$},description={$\pitilde(\thetatilde_i(\xtilde)_{i \in \Zln}) = \thetatilde_i(\xtilde)_{i \in \Zn}$ is the affine lift of $\pi$ to $\Atilde \to \Btilde$.},sort={21}}
\myglossaryentry{glo@pitildei}{name={$\pitilde_i$},description={$\pitilde_i = \pitilde \circ (1,i,0) = \thetatilde_{i+j}(\cdot)_{j \in \Zn}$.},sort={22}}
\myglossaryentry{glo@Ptildei}{name={$\Ptilde_i$},description={$\Ptilde_i = (1,i,0).\zeroA=(\theta_{i+j}(\zeroA))_{j \in \Zln}$.},sort={23}}
\myglossaryentry{glo@Rtildei}{name={$\Rtilde_i$},description={$\Rtilde_i = \pitilde_i(\zeroA)=\pitilde(\Ptilde_i)$.},sort={24}}

\myglossaryentry{glo@Scal}{name={$\Scal$},description={$\Scal=\Zl$ (When $\ell \wedge n=1$)},sort={33}}
\myglossaryentry{glo@ei}{name={$(e_1,\dots,e_g)$},description={A basis of
$\Zln$},sort={31}}
\myglossaryentry{glo@di}{name={$(d_1,\dots,d_g)$},description={$d_i=n e_i$},sort={32}}
\myglossaryentry{glo@Sfrak}{name={$\Sfrak$},description={$\Sfrak=\{d_1, d_2,
\dots, d_g, d_1+d_2, \dots d_1+d_g, d_2+d_3, \dots d_{g-1}+d_g \}$ (When
$\ell \wedge n =1$)},sort={34}}

\myglossaryentry{glo@el}{name={$e'_{\ell}$},description={The extended
commutator pairing on $B_k[\ell]$},sort={40}}
\myglossaryentry{glo@eW}{name={$e_{W}$},description={The Weil
pairing.},sort={41}}
\myglossaryentry{glo@ec}{name={$e_c$},description={The canonical pairing
on $\Zn \times \dZn$.},sort={42}}
\myglossaryentry{glo@Bell}{name={$\Bell$},
description={the affine cone of $(B_k, \bpol, \Thetall)$ where $\bpol=[\ell]^* \ppol$ and $\Thetall$ is a theta
structure on $(B_k, \bpol)$ compatible with $\ThetaB$.},sort={43}}
\myglossaryentry{glo@elliso}{name={$\ellisotilde$},description={$\ellisotilde: \Bell \to \Btilde$ is the morphism lifting $\elliso: B_k \to B_k$.},sort={44}}
\myglossaryentry{glo@chainadd}{name={$\chaineadd$},description={An addition chain},sort={45}}
\myglossaryentry{glo@chainmult}{name={$\chainemultadd$},description={A multiplication chain},sort={50}}

\begin{document}

\maketitle

\begin{abstract} \noindent We describe an efficient algorithm for the
computation of isogenies between abelian varieties represented in the
coordinate system provided by algebraic theta functions. We explain how to
compute all the isogenies from an abelian variety whose kernel is
isomorphic to a given abstract group. We also describe an analog of Vélu's
formulas to compute an isogeny with a prescribed kernel. All our algorithms
rely in an essential manner on a generalization of the Riemann formulas.

In order to improve the efficiency of our algorithms, we introduce a
point compression algorithm that represents a point of level
$4\ell$ of a $g$ dimensional abelian variety using only $g(g+1)/2\cdot 4^g$
coordinates. We also give formulas to compute the Weil and commutator
pairing given input points in theta coordinates. All the algorithms
presented in this paper work in general for any abelian variety defined
over a field of odd characteristic. 
\end{abstract}

\pdfbookmark[1]{\contentsname}{toc}
\tableofcontents

\pdfbookmark[1]{Algorithms}{loalgo}
\section*{List of Algorithms}
\theoremlisttype{all}
\listtheorems{algorithm}
\clearpage
\pdfbookmark[1]{Notations}{glo}
\printglossary[title={List of Notations}]
\clearpage

\section{Introduction}
\label{sec@intro}

In this paper, we are interested in some algorithmic aspects of
isogeny computations between abelian varieties. Computing isogenies
between abelian varieties may be seen as different kind of computational
problems depending on the expected input and output of the algorithm.
These problems are:
\begin{itemize}
  \item Given an abelian variety $A_k$ and an abstract finite abelian group
    $K$ compute all the abelian varieties $B_k$ such that there exists an
    isogeny $A_k \rightarrow B_k$ whose kernel is isomorphic to $K$, and
    compute these isogenies. 
\item Given an abelian variety $A_k$ and a
finite subgroup $K$ of $A_k$, recover the quotient abelian variety
$B_k=A_k/ K$ as well as the isogeny $A_k \rightarrow B_k$.
\item Given two isogenous abelian varieties, $A_k$ and $B_k$, compute
  explicit equations for an isogeny map $A_k \rightarrow B_k$. 
\end{itemize}
Here, we are concerned with the first two problems. In the case that the abelian variety is an
elliptic curve, efficient algorithms have been described that solve all
the aforementioned problems \cite{lercier-algorithmique}. For higher-dimensional abelian varieties much less is known. Richelot's formulas
\cite{Mestre4,Mestre5} can be used to compute $(2,2)$-isogenies
between abelian varieties of dimension~$2$. The paper \cite{0806.2995v2} also introduced a method to compute certain isogenies
of degree $8$ between Jacobian of curves of genus three. In this
paper, we present an algorithm to compute $(\ell, \ldots ,
\ell)$-isogenies between abelian varieties of dimension $g$ for any
$\ell \geq 2$ and $g \geq 1$. Possible applications of our algorithm includes:
\begin{itemize}
  \item The transfer the discrete logarithm from an abelian variety to
    another abelian variety where the discrete logarithm is easy to
    solve \cite{1149.94329} 
  \item The computation of isogeny graph to obtain a description the
    endomorphism ring of an Abelian variety.
\end{itemize}

By Torelli's theorem there is a one on one correspondence between
principally polarized abelian varieties of dimension $2$ and 
Jacobians of genus $2$ hyperelliptic curves. Thus the modular space
of principally polarized abelian varieties of dimension~$2$ is
parametrized by
the three Igusa invariants, and one can define modular polynomials between
these invariants much in the same way as in the genus-one case
\cite{0804.1565v2}. However the height of these
modular polynomials explodes with the order, making their computations impractical: only
those are known~\cite{broker-modular}. In order to circumvent this
problem, in the article
\cite{0910.4668v1}, we have defined a modular correspondence between abelian
varieties in the moduli of marked abelian varieties. This moduli
space is well-suited for computating modular correspondences
since the associated
modular polynomials have their coefficients in $\{1,-1\}$, and there
is no explosion as before. 

In this paper, we explain how, given a solution to this modular
correspondence (provided for instance by the algorithm described in
\cite{0910.4668v1}), one can compute the associated isogeny. Once such a
modular point is obtained, the isogeny can be computed using only simple
addition formulas of algebraic theta functions, so in practice, the
computation of the isogeny takes much less time than the computation of a
point provided by the modular correspondence. Note that this is
similar to the genus-one case. 
For elliptic curves, the computation of a root of the modular
polynomial is not mandatory if the
points in the kernel of the isogeny are given, since this is the input taken by
Vélu's formulas. Here, we explain how to recover the equations
of an isogeny given the points of its kernel, yielding
a generalization of Vélu's formulas.

Our generalization introduces however a difference compared to the usual
genus-$1$ case. For elliptic curves, the modular
polynomial of order $\ell$ give the moduli space of $\ell$-isogenous elliptic curves.
In our generalized setting, the modular correspondence in the coordinate
system of theta null points gives $\ell^g$-isogenous abelian
varieties with a theta structure of different level. As a consequence, a point in this
modular space corresponds to an $\ell^g$-isogeny, together with a
symplectic structure of level $\ell$. Another method would be to describe a
modular correspondence between abelian varieties with theta structures of
the same level, see~\cite{0910.1848v1} for an example with $\ell=3$ and $g=2$.


The paper is organized as follow. In Section~\ref{sec@algo} we recall
Vélu's formulas and outline our algorithms. In Section~\ref{sec@modular},
we recall the definition of the modular correspondence given in
\cite{0910.4668v1}, and we study the relationship between isogenies and the
action of the theta group. We
recall the addition relations, which play a central role in this paper
in Section~\ref{sec@addition}. We then explain how to compute the
isogeny associated to a modular point in Section~\ref{sec@isogenies}.
If the
isogeny is given by theta functions of level $4 \ell$, it
requires $(4 \ell)^g$ coordinates. We give a point compression algorithm in
Section~\ref{subsec@compression}, showing how to express such an isogeny with
only $g(g+1)/2\cdot 4^g$ coordinates. In Section~\ref{sec@velu} we give a
full generalization of Vélu's formulas that constructs an isogenous modular
point with prescribed kernel. This algorithm is
more efficient
than the special Gr\"obner basis
algorithm from \cite{0910.4668v1}. There is a strong connection between isogenies and
pairings, and we use the above work to explain how one can compute the
commutator pairing and how it relates to the usual Weil pairing in
Section~\ref{sec@pairings}.

\section{Computing Isogenies}
\label{sec@algo}
In this section, we recall how one can compute isogenies between elliptic
curves. We then outline our algorithm to compute isogenies between abelian
varieties.

\subsection{Elliptic curves and Vélu's formulas}
\label{subsec@algo1}
Let $(E_k, \zeroE)$ be an elliptic curve given by a Weierstrass equation
$y^2=f(x)$ with $f$ a degree-$3$ monic polynomial. 
V\'elu's formulas rely on the intrinsic characterization of the coordinate
system $(x,y)$ giving the Weierstrass model of $E_k$ as:
\begin{equation}
      \begin{aligned}
        v_{\zeroE}(x)=-3 & \qquad v_{P}(x) \geq 0 \text{\quad if $P \ne \zeroE$} \\
        v_{\zeroE}(y)=-2 & \qquad v_{P}(y) \geq 0 \text{\quad if $P \ne \zeroE$}  \\
        y^2/x^3(\zeroE) = 1, & 
      \end{aligned}
  \label{eq@velu}
\end{equation}
where $v_Q$ denotes the valuation of the local ring of $E_k$ in the closed
point $Q$.

      \begin{theorem}[Vélu]
    Let $G \subset E_k(k)$ be a finite subgroup. Then $E_k/G$ is given by
   $Y^2=g(X)$ with $g$ a degree $3$ monic polynomial where
    \begin{gather*}
      X(P) = x(P) + \sum_{Q \in G\setminus \left\{ \zeroE \right\}} x(P+Q) - x(Q) \\
      Y(P) = y(P) + \sum_{Q \in G\setminus \left\{ \zeroE \right\}} y(P+Q) - y(Q)
    \end{gather*}
        \label{th@velu}
      \end{theorem}

    \begin{proof}
      Indeed, $X$ and $Y$ are in $k(E_k)^G$, and it is easily seen that they satisfy the
      relations~\eqref{eq@velu}.
    \end{proof}
A consequence of the that theorem is that,
given a finite subgroup $G$ of cardinality $\ell$ of an
elliptic curve $E_k$  an equation $y^2=f(x)$ with $f$ a degree
$3$ polynomial, it is
possible to compute the
Weierstrass equation of the quotient $E_k /G$ at the cost of $O(\ell)$ additions in $E_k$.

The modular curve $X_0(\ell)$ parametrizes the set of isomorphism
classes of elliptic curves together with a $\ell$-torsion subgroup.
For instance $X_0(1)$ is just the line of $j$-invariants.  Let
$\Phi_{\ell}(x,y) \in \Z[x,y]$ be the order $\ell$ modular polynomial.
It is well known that  the roots of  $\Phi_{\ell}(j(E_k), \cdot)$ give the
$j$-invariants of the elliptic curves $\ell$-isogenous to $E_k$. Since an
$\ell$-isogeny is given by a finite subgroup of $E_k$ of order $\ell$,
we see that $\Phi_{\ell}(x,y)$ cuts out a curve isomorphic to $X_0(\ell)$ in $X_0(1) \times X_0(1)$.

Given an elliptic curve $E_k$ with $j$-invariant $j_{E_k}$, the computation of
isogenies can be done in two steps: 
\begin{itemize}
  \item
    First, find the solutions of $\Phi_\ell(j_{E_k},X)$ where $\Phi_\ell(X,Y)$ is
    the order $\ell$ modular polynomial; then recover from a root
    $j_{E'_k}$ the
equation of the corresponding curve $E'_k$ which is $\ell$-isogenous to $E_k$.
\item Next,
  using Vélu's formulas,
compute the isogeny $E_k \to E'_k$. 
\end{itemize}
For some applications such as isogeny-graph computation, only the
first step is required, while for other applications it is necessary to obtain the explicit equations describing the isogeny. Note that the first
step is unnecessary if one already know the points in the kernel of the isogeny.

\subsection{Isogenies on abelian varieties}
\label{subsec@algo2}
Let $A_k$ be an abelian variety of
dimension~$g$ over a field~$k$ and denote by $K(A_k)$ its function
field. An isogeny is a finite surjective map of abelian varieties. In
the following we only consider separable isogenies i.e. isogenies
$\pi: A_k \to B_k$ such that the function field $K(A_k)$ is a finite
separable extension of $K(B_k)$. A separable isogeny is uniquely
determined by its kernel, which is a finite subgroup of $A_k(\kbar)$. In
that case, the cardinality of the kernel is the degree of the isogeny.
In the rest of this paper, by $\ell$-isogeny for $\ell >0$, we always mean a
$(\ell,\cdots,\ell)$-isogeny where $(\ell,\cdots,\ell) \in \Z^g$.

We have seen that it is possible to define a modular correspondence
between the Igusa invariants, parameterizing the set of dimension $2$
principally polarized abelian varieties, but the coefficients explosion of
the related modular polynomials makes it computationally inefficient.
In order to
mitigate this problem and obtain formulas suitable for general 
$g$-dimensional abelian varieties, we use the
moduli space of marked abelian varieties. 

Let $g \in \Nstar$ and let $n \in \N$ be such that $2 |n$. Let
$\overn=(n,n,\dots,n) \in \Z^g$, and $\Zn=\Z^g/n\Z^g$. We denote $\Mn$
the modular space of marked abelian varieties $(A_k,\pol,\ThetaA$)
where $\pol$ is a polarization and $\ThetaA$ is 
symmetric theta structure $\ThetaA$ of type $\Zn$
(see~\cite{MR34:4269}). The forgetting map $(A_k,\pol,\ThetaA) \mapsto
(A_k,\pol)$ is a finite map from $\Mn$ to the moduli space of abelian
varieties with a polarization of type $\Zn$.

\glsadd{glo@Mn}
We recall~\cite{MR36:2621} that if $4 |n$, then $\Mn$ is open in the projective variety described by
the following equations in $\Pvar(k(\Zn))$:
\begin{equation}
  \begin{gathered}
  \begin{multlined}[0.80\columnwidth]
\big(\sum_{t \in \Ztwo} \chi(t) a_{x+t} a_{x+t}\big).\big(\sum_{t \in \Ztwo} \chi(t)
a_{u+t} a_{u+t}\big) 
= \\
\big(\sum_{t \in \Ztwo} \chi(t) a_{z-x+t} a_{z-y+t}\big).\big(\sum_{t \in \Ztwo} \chi(t) a_{z-u+t} a_{z-v+t}\big)
  \end{multlined} \\
a_{x}=a_{-x} 
  \end{gathered}
  \label{eq@equations_Mn}
\end{equation}
for all $x,y,u,v \in \Zn$, such that $x+y+u+v=2z$ and all $\chi \in \dZtwo$.

In~\cite{0910.4668v1}, we have described a modular correspondence $\phi: \Mln \to \Mn
\times \Mn$ for $\ell \in \Nstar$, which can be seen as a generalization
of the modular correspondence $X_0(\ell) \to X_0(1) \times X_0(1)$ for
elliptic curves.
Let $p_1$ and $p_2$ be the corresponding projections $\Mn \times \Mn \to
\Mn$, and let $\phi_1=p_1 \circ \phi$, $\phi_2=p_2 \circ \phi$.
The map $\phi_1: \Mln \to \Mn$ is such that $(x,
\phi_1(x))$ are modular points corresponding to $\ell$-isogenous varieties
We recall that $\phi_1$ is defined by 
$\phi_1\left( (a_i)_{i \in \Zln} \right) = (a_i)_{i \in \Zn}$ where
$\Zn$ is identified as a subgroup of $\Zln$ by the map $x \mapsto \ell
x$. 

Suppose that we are given a modular point $(b_i)_{i \in \Zn}$ corresponding to the
marked abelian 
variety $(B_k, \pol_0, \ThetaB)$. If $B_k$ is the Jacobian variety of an hyperelliptic curve, one
may recover the associated modular point for $n=4$ via Thomae formulas~\cite{MR86b:14017}.

Suppose for now that $4 |n $ and that $\ell$ is prime to $n$.
Our algorithm works in two steps:
\begin{enumerate}
  \item \textbf{Modular computation} Compute a modular point
    $(a_i)_{i \in \Zln} \in \phi_1^{-1}\left( (b_i)_{i \in \Zn} \right)$.
    This can be done via the specialized Gr\"obner basis algorithm described
    in \cite{0910.4668v1}, but see also Section~\ref{sec@velu} for a more
    efficient method.
  \item \textbf{Vélu's like formulas} Use the addition formula in $B_k$ to
    compute the isogeny $\hat{\pi}: B_k \to A_k$ associated to the modular point
    solution. Here $(A_k, \pol, \ThetaA)$ is the marked abelian variety corresponding to 
    $(a_i)_{i \in \Zln}$. This step is described in
    Section~\ref{sec@isogenies}. 
\end{enumerate}

We can also compute an isogeny given by its kernel $K$ by using the results of Section~\ref{subsec@velu}
to construct the corresponding modular point $(a_i)_{i \in \Zln}$ from $K$.
We thus have a complete generalization of Vélu's formulas for higher
dimensional abelian varieties
since the reconstruction of the modular point $(a_i)_{i \in \Zln}$ from the
kernel $K$ only requires the addition formulas in $B_k$ (together with the
extraction of 
$\ell^{th}$-roots). In Section~\ref{subsec@ComputingModular} we
explain how to use this to speed up Step~1 of our algorithm (we call this
Step~1').

If the kernel of the isogeny is unknown, the most time-consuming part of our
algorithm is the computation of a maximal subgroup of rank $g$ of the
$\ell$-torsion, which means that currently, with $g=2$ we can go up to
$\ell=31$ relying on the current state-of-the-art implementation~\cite{gaudryrecord}.
In order to speed up Step~$2$, which requires $O(\ell^g)$ additions to be
performed in
$B_k$, and compute with a compact representation,
it is important to consider the smallest possible $n$.
If $n=2$, we cannot prove that the modular system to be solved in
Step~1 is of dimension~$0$. However Step~1', which is faster, does not
require a modular solution but only the kernel of the isogeny, so our
algorithm works with $n=2$ too. Note, however, that some care must be taken when
computing additions on $B_k$, since the algebraic theta functions only give an
embedding of the Kummer variety of $B_k$ for $n=2$.

For an actual implementation the case $n=2$ is critical (it allows
for a more compact representation of the points than $n=4$: we gain a
factor $2^g$, it allows for a faster addition chain,
see Section~\ref{subsubsec@compressed_addition}, but most importantly it reduces the
most consuming part of our algorithm, the computation of the points of
$\ell$-torsion, since there are half as much such points on the Kummer variety).
For each algorithm that we use, we give an explanation on how to adapt it
for the level~$2$ case: see Section~\ref{subsubsec@addition_level2} and the
end of Sections~\ref{subsec@isogeny}, \ref{subsec@velu},
\ref{subsec@ComputingModular} and~\ref{subsec@compute_commutator}.

The assumption that $n$ is prime to $\ell$ is not necessary either but there is one
important difference in this case. Suppose that we are given $B_k[\ell]$.
Since $B_k$ is given by a theta structure of level~$n$, we also have
$B_k[n]$. If $\ell$ is prime to $n$, this gives us $B_k[\ell n]$, and we can
use Step~1' to reconstitute a modular point of level $\ell n$. If $\ell$ is
not prime to $n$, we have to compute $B_k[\ell n]$ directly.  

It is also possible to compute more general types of isogenies via our algorithm. With
the notations of Section~\ref{sec@modular}, let $\delta_0=(\delta_1,
\ldots, \delta_g)$ be a sequence of integers such that $\delta_i |
\delta_{i+1}$, and let $(b_i)_{i \in \Zdeltao} \in \Mstruct{\delta_0}$ be a
modular point corresponding to an abelian variety $B_k$.  Let
$\delta'=(\ell_1, \ldots, \ell_g)$ (where $\ell_i | \ell_{i+1}$) and define
$\delta=(\delta_1 \ell_1, \ldots, \delta_g \ell_g)$. Let $(a_i)_{i \in
\Zdelta} \in \Mdelta$ be such that $\phi\left( (a_i)_{i \in \Zdelta}
\right) = (b_i)_{i \in \Zdeltao}$. The theta null point $(a_i)_{i \in \Zdelta}$ corresponds to
an abelian variety $A_k$, such that there is a
$(\ell_1,\cdots,\ell_g)$-isogeny $\pi: A_k \to B_k$, which can be computed
by the isogeny theorem~\cite{MR34:4269} (see
Section~\ref{subsec@isogeniethm}). The isogeny we compute in Step~2 is the contragredient isogeny $\pidual: B_k \to A_k$
of type $(\ell_g / \ell_1,\ell_g / \ell_2, \cdots, 1, \ell_g, \ell_g, \cdots, \ell_g)$. 
Using the modular correspondence $\phi$ to go back to a modular point of
level $\delta_0$ (see Section~\ref{subsec@modcorr}) gives an isogeny of
type $(\ell_g / \ell_1,\ell_g / \ell_2, \cdots, 1, \ell_1 \ell_g, \ell_2 \ell_g, \cdots, \ell_g \ell_g)$. 
For
the clarity of the exposition, we will stick to the case $\delta_0=\overn$
and $\delta=\overln$ and we leave to the reader the easy generalization.

Let us make some remarks on our algorithm. First note that to compute
$\ell$-isogenies, we start from a theta null point of level $n$ to get a
theta null point of level $\ell n$. We can then go back to a point of level
$n$ (see Section~\ref{subsec@modcorr}), but in this case we are computing
$\ell^2$-isogenies. A second remark is that all our computations are
geometric, not arithmetic, since the projective embedding given by theta
functions of level~$\ell n$ is not rational. A last remark is that since we
use different moduli spaces, our method is not a straight-up generalization
of the genus-$1$ case. In particular, computing a modular point solution
$(a_i)_{i \in \Zln}$ is the same as choosing an $\ell$-isogeny \textbf{and}
a theta structure of level~$\ell$, so there are many more modular solutions
than there is $\ell$-isogenies. Hence, as noted in the introduction, the
most efficient method in our cases is to compute the points of
$\ell$-torsion to reconstitute the modular point. 


\section{Modular correspondences and theta null points}
\label{sec@modular}
In this section, we recall some results of~\cite{0910.4668v1} and 
notations that we will use in the rest of the paper.
In Section~\ref{subsec@thetastruct} we recall the definition of a theta
structure and the associated theta functions~\cite{MR34:4269}. In
Section~\ref{subsec@isogeniethm} we recall the isogeny theorem, which give
a relations between the theta functions of two isogenous abelian varieties.
In Section~\ref{subsec@modcorr} we explain the modular
correspondence defined in~\cite{0910.4668v1}. In
Section~\ref{subsec@action} we study the connection between isogenies and the
action of the theta group.

\subsection{Theta structures}
\label{subsec@thetastruct}

\glsadd{glo@Ak}
\glsadd{glo@Kpol}
\glsadd{glo@Zn}
Let $A_k$ be a $g$ dimensional abelian variety over a field $k$. Let
$\pol$ be a degree-$d$ ample symmetric line bundle on $A_k$. We
suppose that $d$ is prime to the characteristic of $k$ or that $A_k$
is ordinary.  Denote by $K(\pol)$ the kernel of the isogeny
$\phi_{\pol} : A_k \rightarrow \hat{A}_k$, defined on geometric points
by $x \mapsto \tau_x^* \pol \otimes \pol^{-1}$ where $\tau_x$ is the
translation by $x$. Let $\delta=(\delta_1, \ldots, \delta_g)$ be the
sequence of integers satisfying $\delta_i | \delta_{i+1}$ such that, as 
group schemes $K(\pol) \simeq \bigoplus_{i=1}^g (\Z / \delta_i
\Z)^2_k$. We say that $\delta$ is the type of $\pol$. In the following
we let $\Zdelta = \bigoplus_{i=1}^g (\Z / \delta_i \Z)_k$, $\dZdelta$
be the Cartier dual of $\Zdelta$, and $\Kdelta = \Zdelta \times
\dZdelta$.

\glsadd{glo@Gpol}
\glsadd{glo@Hdelta}
\glsadd{glo@sKpola}
\glsadd{glo@ec}
Let $G(\pol)$ and $\Hdelta$ be respectively the theta group of $(A_k,
\pol)$ and the Heisenberg group of type $\delta$ \cite{MR34:4269}. In
this article, elements of $G(\pol)$ will be written as $(x,\psi_x)$
with $x \in K(\pol)$ and $\psi_x : \pol \rightarrow \tau^*_x \pol$ is
an isomorphism.  We know that $G(\pol)$ and $\Hdelta$ are central
extensions of $K(\pol)$ and $K(\delta)$ by $\mathbb{G}_m$. By definition, a theta
structure $\ThetaA$ on $(A_k,\pol)$ is an isomorphism of central
extensions from $\Hdelta$ to $G(\pol)$. We denote by $\epol$ the
commutator pairing \cite{MR34:4269} on $K(\pol)$ and by $\edelta$ the
canonical pairing on $\Zdelta \times \dZdelta$ (We often drop the indice
$\delta$ in $e_{\delta}$ when there is no risk of confusion). We remark
that a theta
structure $\ThetaA$ induces a symplectic isomorphism
$\ThetabarA$ from $(\Kdelta, \edelta)$ to $(\Kpol, \epol)$. We
denote by $\Kpol = \Kpola \times \Kpolb$ the decomposition into
maximal isotropic subspaces induced by $\ThetabarA$. The sections
$\Zdelta \rightarrow \Hdelta$ and $\dZdelta \rightarrow \Hdelta$
defined on geometric points by $x \mapsto (1,x,0)$ and $y \mapsto
(1,0,y)$ can be transported by the theta structure to obtain 
natural sections $s_{\Kpola}: \Kpola \to \Gpol$ and $s_{\Kpolb}: \Kpolb \to \Gpol$ of the canonical
projection $\kappa : G(\pol) \to \Kpol$. Recall \cite[pp.
291]{MR34:4269} that a level subgroup $\tilde{K}$ of $G(\pol)$ is a
subgroup such that $\tilde{K}$ is isomorphic to its image by $\kappa$.
We define the maximal level subgroups $\tilde{K}_1$ over $K_1(\pol)$
and $\tilde{K}_2$ over $K_2(\pol)$ as the image by $\ThetaA$ of
the subgroups $(1,x,0)_{x \in \Zdelta}$ and $(1,0,y)_{y \in \dZdelta}$
of $\Hdelta$. 

\glsadd{glo@thetai}
Let $V=\Gamma(A_k, \pol)$. The theta group $G(\pol)$ acts on $V$ by $v
\mapsto \psi_x^{-1}\tau^*_x(v)$ for $v \in V$ and $(x,\psi_x) \in
G(\pol)$. This action can be transported via $\ThetaA$ to an
action of $\Hdelta$ on $V$. It can be shown that there is a unique (up
to a scalar factors) basis $(\vartheta_{i})_{i \in \Zdelta}$ of $V$
such that this action is given by:
\begin{equation}
  (\alpha,i,j). \vartheta^{\ThetaA}_{h} =\alpha. \edelta(-i-h,j).
  \vartheta^{\ThetaA}_{h+i}.
\label{eq@actiontheta}
\end{equation}
(see \cite{MR36:2622,MR2062673}  for a connection between
algebraic theta functions and the classical, analytic theta functions.)
If there is no ambiguity, in this paper, we will sometimes drop the superscript
$\ThetaA$ in the notation $\theta_k^{\ThetaA}$.
We briefly recall the construction of this basis: let $\Az_k$ be the
quotient of $A_k$ by $K_2(\pol)$ and $\pi : A_k \rightarrow \Az_k$ be
the natural projection.  By Grothendieck descent theory, the
data of $\tilde{K}_2$ is equivalent to the data of a couple $(\pol_0,
\lambda)$ where $\pol_0$ is a degree-one ample line bundle on $\Az_k$
and $\lambda$ is an isomorphism $\lambda : \pi^*(\pol_0) \rightarrow
\pol$. Let $s_0$ be the unique global section of $\pol_0$ up to a
constant factor and let $s = \lambda(\pi^*(s_0))$. We have the
following proposition (see \cite{MR34:4269})
\begin{proposition}\label{prop@thetadef}
  For all $i \in \Zdelta$, let $(x_i, \psi_i)=
  \ThetaA((1,i,0))$.  We set $\vartheta^{\ThetaA}_i =
  (\psi_x^{-1}\tau_x^*(s))$.  
  The elements
  $(\vartheta^{\ThetaA}_i)_{i \in \Zdelta}$ form a basis of the
  global sections of $\pol$; it is uniquely determined (up to a
  multiplicative factor independent of $i$) by 
  $\ThetaA$.
\end{proposition}

\glsadd{glo@zeroA}
This basis gives a projective embedding $\phi_{\ThetaA}: A_k \to
\Pvar^{d-1}_k$ which is uniquely defined by the theta structure
$\ThetaA$. The point $(a_i)_{i \in \Zdelta} :=
\phi_{\ThetaA}(0_{A_k})$ is called the theta null point associated to
the theta structure.  Mumford proves \citep*{MR34:4269} that if
$4|\delta$, $\phi_{\ThetaA}(A_k)$ 
is the closed subvariety of $\proj_k^{d-1}$ defined by the homogeneous ideal
generated by the Riemann equations:
\begin{theorem}[Riemann equations]
  \label{prop@riemann}
  For all $x,y,u,v \in \Zstruct{2 \delta}$ that are
  congruent modulo $\Ztwo$, and all $\chi \in \dZtwo$, we have
\begin{multline}
\big(\sum_{t \in \Ztwo} \chi(t) \vartheta_{x+y+t} \vartheta_{x-y+t}\big).\big(\sum_{t \in \Ztwo} \chi(t)
a_{u+v+t} a_{u-v+t}\big)= \\
= \big(\sum_{t \in \Ztwo} \chi(t) \vartheta_{x+u+t} \vartheta_{x-u+t}\big).\big(\sum_{t \in \Ztwo} \chi(t)
a_{y+v+t} a_{y-v+t}\big). 
\label{eq@equations_Ak}
\end{multline}
\end{theorem}

\glsadd{glo@Atilde}
\glsadd{glo@thetatildei}
Let $p_{\Avar_k(V)}: \Avar_k(V) \to \Pvar_k(V)$ be the canonical
projection. Let $\Atilde=p_{\Avar_k(V)}^{-1}(A_k)$ be the affine cone of
$A_k$ and denote by $p_{A_k}:\Atilde \to A_k$, the application induced
by $p_{\Avar_k(V)}$. Since this affine cone will play a central role
in this paper, we take the following convention: if $(\theta_i)_{i \in
  \Zdelta}$ is a homogeneous coordinate system on $\Pvar(V)$ then we
denote by $(\thetatilde_i)_{i \in \Zdelta}$ the associated affine
coordinate system of $\Avar(V)$. For
instance, $p_{A_k}$ is given by $(\thetatilde_i(x))_{i \in \Zdelta}
\mapsto (\theta_i(x))_{i \in \Zdelta}$. It should be remarked that, for
$i \in \Zdelta$, $\thetatilde_i$ is a well defined function on
$\Atilde$ and for any geometric point $x \in \Atilde$, we denote by $\thetatilde_i(x)$ 
its values in $x$.

\glsadd{glo@rhopola}
\glsadd{glo@rhopol}
Since the action of $\Gpol$ on $V$ is affine, the
action~\eqref{eq@actiontheta} gives an action $\rhopola^*$ on
$\Atilde$. This action descends to a projective action $\rhopol^*$ of
$\Kpol$ on $A_k$ which is simply the action by translation.
We will use the same notations for the action of
$\Hdelta$ (resp. of $\Kdelta$) induced by $\ThetaA$. 


\subsection{Isogenies compatible with a theta structure}
\label{subsec@isogeniethm}

\glsadd{glo@Bk}
\glsadd{glo@zeroB}
\glsadd{glo@pi}
Let $\delta' \in \Z^g$ be such that $2|\delta' | \delta$, and write
$\delta=\delta' \cdot \delta''$. In the following we consider
$\Zdeltab$ as a subgroup of $\Zdelta$ via the map $\phi:
(x_i)_{i \in [1..g]} \in \Zdeltab \mapsto (\delta''_i x_i)_{i \in
  [1..g]} \in \Zdelta$. From now on, when we write $\Zdeltab \subset \Zdelta$, we
  always refer to this map. Let $K$ be the subgroup
$\ThetabarA(\dZstruct{\delta''})$ of $\Kpolb$ and let $\pi_K$ be
the isogeny $A_k \to B_k=A_k/K$. By Grothendieck descent theory, the level
subgroup $\tilde{K}:=s_{\Kpolb}(K)$ induces a polarization $\pol_0$ on
$B_k$, such that $\pol \simeq \pi_K^{*}(\pol_0)$. The theta group
$G(\pol_0)$ is isomorphic to $\Ze(\tilde{K}) / \tilde{K}$ where
$\Ze(\tilde{K})$ is the centralizer of $\tilde{K}$ in $G(\pol)$
\cite{MR34:4269}. Let
$\ThetaB$ be the unique theta structure on $B_k$ compatible with
the theta structure on $A_k$ \cite[Sec. 3]{0910.4668v1}. We have
\cite[Prop. 4]{0910.4668v1}:

\begin{proposition}[Isogeny theorem for compatible theta structures]\label{prop@iso}
  Let $\phi: \Zdeltab \to \Zdelta$ be the canonical embedding. Let
  $(\vartheta^{\ThetaA}_i)_{i \in \Zdelta}$ 
  (resp. $(\vartheta^{\ThetaB}_i)_{i \in \Zdeltab}$) be the
  canonical basis of $\pol$ (resp. $\pol_0$) associated to
  $\ThetaA$ (resp. $\ThetaB$). There exists some $\omega \in
  \overline{k}^*$ such that for all $i \in \Zdeltab$
  \begin{equation}
    \pi_K^*(\vartheta^{\ThetaA}_i) = \omega
    \vartheta^{\ThetaB}_{\phi(i)}.
    \label{eq@isotheorem}
  \end{equation}
  In particular, the theta null point of $B_k$ is given by
  \begin{equation}
  (b_i)_{i \in \Zdeltab} = (a_{\phi(i)})_{i \in \Zdeltab}
    \label{eq@isotheorem0}
  \end{equation}
\end{proposition}
The above proposition is a particular case of the more general
isogeny theorem \cite[Th. 4]{MR34:4269}.  

\glsadd{glo@Btilde}
\glsadd{glo@pitilde}
On the affine cones this proposition shows that, given
the theta null point $(a_i)_{i \in \Zdelta}$, the morphism
\begin{align*}
  \tilde{\pi}_{K}: \tilde{A}_k &\to \tilde{B}_k \\
( \tilde{\theta} )_{i \in \Zdelta} &\mapsto 
( \tilde{\theta} )_{i \in \phi(\Zdeltab)} 
\end{align*}
makes the following diagram commutative:

\begin{center}
\begin{tikzpicture}[auto, 
  text height=1.5ex, text depth=0.25ex, 
  myarrow/.style={->,shorten >=1pt}]
    \matrix (A_k) [row sep=1.5cm, column sep=1.5cm, matrix of math nodes]
    { \tilde{A}_k &  A_k \\
    \tilde{B}_k &  B_k \\} ;
    \draw[myarrow] (A_k-1-1) to node {$p_{A_k}$} (A_k-1-2);
    \draw[myarrow] (A_k-2-1) to node {$p_{B_k}$} (A_k-2-2);
    \draw[myarrow] (A_k-1-1) to node {$\tilde{\pi}$} (A_k-2-1);
    \draw[myarrow] (A_k-1-2) to node {${\pi}$} (A_k-2-2);
\end{tikzpicture}
\end{center}

For the sake of simplicity, we set from now on $\deltao=\overn$, and
$\deltab=\overl$ so that $\delta=\overln$ (we stated Proposition
(\ref{prop@iso}) in a more general form because we can use it to compute the $\ell$-torsion
on $B_k$, see Section~\ref{subsec@ComputingModular}). Let $(b_i)_{i \in \Zn}$ be a
theta null point associated to a triple $(B_k, \pol_0, \ThetaB)$; we want
to compute an $\ell$-isogeny $B_k \to A_k$.
Since $n$ is fixed, we cannot apply the isogeny theorem directly since it
requires $\ell |n$. However, if $(a_i)_{i \in \Zln} \in \Mln$ is a theta
null point corresponding to a triple $(A_k, \pol, \ThetaA)$ where
the theta structure $\ThetaA$ is compatible with $\ThetaB$ (this is
equivalent to $(a_i)_{i \in \Zln}$ satisfying \eqref{eq@isotheorem0}),  
then
Proposition~\ref{prop@iso} gives an (explicit) isogeny $\pi: A_k \to B_k$.
So to the modular point $(a_i)_{i \in \Zln}$ we may associate the isogeny
$\pidual: B_k \to A_k$ and this is the isogeny we compute in Step~2 of our
algorithm (Section~\ref{subsec@algo2}).

We have the following diagram
\begin{center}
\begin{tikzpicture}[scale=2.5,auto, 
  text height=1.5ex, text depth=0.25ex, 
  myarrow/.style={->,shorten >=1pt}]
  \node[anchor=west] (B) at (0.5,-0.87) {$y \in B_k$} ;
  \node[anchor=west]  (A2) at (1,0) {$z \in A_k$} ;
  \node[anchor=west]  (A1) at (0,0) {$x \in A_k$} ;
  \draw[myarrow] (B) to node[swap] {$\pidual$} (A2);
  \draw[myarrow] (A1) to node[swap] {$\pi$} (B);
  \draw[myarrow] (A1) to node {$[\ell]$} (A2);
\end{tikzpicture}
\end{center}
This diagram shows that it is possible to obtain an explicit description of
the rational map $z=\pidual(y)$ by eliminating the variables $x$ in the
ideal generated by $(y=\pi(x)$, $z=\ell \cdot x)$. This can be done using
a Gr\"obner basis algorithm. In Section~\ref{sec@isogenies} gives a
much faster algorithm which uses the
addition formulas in $B_k$ to find the equations of $\pidual$ directly.

\subsection{Modular correspondences}
\label{subsec@modcorr}

In the previous section, we have shown how to compute an isogeny with a
prescribed kernel $K \subset K_2(\pol)[\ell]$ that is isotropic for the
commutator pairing.
Now let $K \subset A_k[\ell]$ be any isotropic subgroup such that we can
write $K= K_1 \times K_2$ with $K_i \subset K_i(\pol)$. Let $B_k=A_k/K$ and
$\pi$ be the associated isogeny. Let $\ThetaB$ and $\ThetaA$ be
$\pi$-compatible theta structures in the sense of Mumford~\cite{MR34:4269}. 
We briefly explain this notion:
if two abelian varieties $(A_k, \pol, \ThetaA)$ and $(B_k,
\pol_0, \ThetaB)$ have $\pi$-compatible marked theta structures, it means that
we have $\pi^{\ast}(\pol_0)=\pol$. By Grothendieck descent
theory, this define a level subgroup $\tilde{K} \subset G(\pol)$ of the
kernel of $\pi$. We have seen in Section~\ref{subsec@isogeniethm} that we
have $G(\pol_0)=\Ze(\Ktilde)/\Ktilde$ where $\Ze(\Ktilde)$ is the
centralizer of $\Ktilde$. The structures $\ThetaA$ and $\ThetaB$ are said to be compatible
if they respect this isomorphism.
The isogeny theorem (\cite[Theorem 4]{MR34:4269}) then gives
a way to compute $(\pi^*(\vartheta_i^{\ThetaB}))_{i \in \Zn}$ given
$(\vartheta_i^{\ThetaA})_{i \in \Zln}$. 

We recall briefly how this works: if $K_1=0$, we say that $\pi$ is an
isogeny of type~$1$, and if $K_2=0$ that $\pi$ is an isogeny of type~$2$. 
In the paper~\cite{0910.4668v1} we have studied the case of isogenies of
type~$1$ and~$2$;
in fact, the notion of compatible isogenies we had defined in these cases
is nothing but
a particular case of the notion of compatible
isogenies described above.
Obviously, by composing isogenies, we only need to study the case of
compatible theta structures between isogenies of type~$1$ or~$2$. We have
already seen the case of isogenies of type~$1$ in the previous Section. Now
let $\Inv_0$ be the automorphism of the Heisenberg group $\Hln$ that permutes
$\Zln$ and $\dZln$: $\Inv_0(\alpha,x,y)=(\alpha,y,x)$. 
We define $\Inv_{A_k}= \ThetaA \circ \Inv_0 \circ \ThetaA^{-1}$, where $\Inv_{A_k}$ is
the automorphism of the Theta group of $A_k$ that permutes $\Kpola$ and
$\Kpolb$. (There is a similar automorphism $\Inv_{B_k}$ of the theta group of
$B_k$; we will usually note these automorphisms $\Inv$ since the theta
group is clear from the context.)
If $\pi_2$ is a compatible
isogeny of type~$2$ between $(A_k,\pol, \ThetaA)$ and $(B_k, \pol_0,
\ThetaB)$, then $\pi_2$ is a compatible isogeny of type~$1$ between 
$(A_k,\pol, \Inv_{A_k} \circ \ThetaA)$ and $(B_k, \pol, \Inv_B \circ \ThetaB)$.

Since the action of $\Inv$ is given by~\cite[Section 5]{0910.4668v1}
\begin{equation}
  \theta_i^{\Inv_{A_k} \circ \ThetaA}= \sum_{j \in \dZln} e(i,j)
\theta_j^{\ThetaA}, 
  \label{eq@inversion_action}
\end{equation}
we see that we have for all $i \in \Zn$
\[ \pi^*  (\vartheta^{\ThetaB}_i)  =  \sum_{j \in \Zl}
\vartheta^{\ThetaA}_{i+n j}.\]

In the following, we focus on isogenies of type~$1$ but
it is easy to adapt the following
to isogenies of type~$2$ (and hence more generally to compatible isogenies
between theta structures) using the action of $\Inv$. Considering both types of isogenies can be useful, see
Section~\ref{subsec@theta_and_addition} or
Section~\ref{subsec@ComputingModular}. The modular correspondence described in Section~\ref{subsec@algo2} is given by 
\[\phi: \Mln \to \Mn \times \Mn, (a_i)_{i \in \Zln} \mapsto ((a_i)_{i
\in \Zn}, (\sum_{j \in \Zl} a_{i+nj})_{i \in \Zn}). \]
Let $\phi_1$ (resp. $\phi_2$) be the composition of $\phi$ with the first
(resp. second) projection.
Let $(a_i)_{i \in \Zln}$ be the theta null point of $(A, \pol,
\ThetaA)$, and put $(b_i)_{i \in \Zl}=\phi_1\left( (a_i)_{i \in \Zl} \right)$, 
and $(c_i)_{i \in \Zl}=\phi_2\left( (a_i)_{i \in \Zl} \right)$. Then
$(b_i)_{i \in \Zl}$ is the theta null point corresponding to the variety
$B_k=A_k / K_2(\pol)[\ell]$, and $(c_i)_{i \in \Zl}$ corresponds to $C_k=A_k /
K_1(\pol)[\ell]$.

The following diagram shows that the composition
$\pi_2 \circ \pidual: B_k \to C_k$ is an $\ell^2$-isogeny:
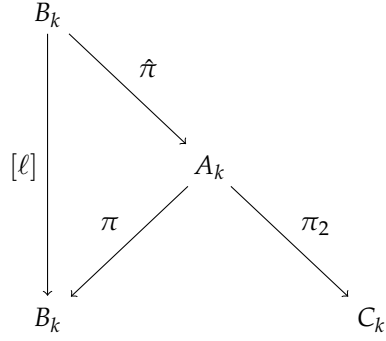
\begin{center}
\begin{tikzpicture}[mydiag]
    \matrix (W) [mymatrix]
    { B_k \&    \&     \\
          \& A_k \&    \\
      B_k \&     \& C_k \\
     } ;
  \draw[myarrow] (W-1-1) to node[swap] {$[\ell]$} (W-3-1);
  \draw[myarrow] (W-1-1) to node {$\pidual$} (W-2-2);
  \draw[myarrow] (W-2-2) to node[swap] {$\pi$} (W-3-1);
  \draw[myarrow] (W-2-2) to node {$\pi_2$} (W-3-3);
  \end{tikzpicture}
\end{center}


\subsection{The action of the theta group on the affine cone and isogenies}
\label{subsec@action}

For the rest of the article, we suppose given an
abelian variety with a theta structure $(B_k, \pol_0, \ThetaB)$ with
associated theta null point $(b_i)_{i \in \Zn}$, and a valid theta null point $(a_i)_{i \in \Zln}$ associated to a triple $(A_k, \pol, \ThetaA)$ such that $\ThetaB$ and $\ThetaA$ are
compatible \cite{0910.4668v1}.  Let $\pitilde: \tilde{A}_k \to
\tilde{B}_k$ be the morphism such that
$\pi^*(\vartheta^{\ThetaB}_i)=\vartheta^{\ThetaA}_i$ for $i \in \Zn$. Note that
the isogeny $\pi : A_k \rightarrow B_k$
lifts to the affine cone as $\pitilde$. 


Let $\{ e_i \}_{i \in [1..g]}$ be the canonical ``basis'' of $\Zln$, 
and $\{ f_i\}_{i \in [1..g]}$ be the canonical ``basis'' of $\dZln$ so
that $\{ e_i, f_i\}_{i \in [1..g]}$ is the canonical symplectic basis of
$\Kln$. For $i \in \Zln$, we let $P_i=\ThetabarA\left( i,0 \right))$ and
for $i \in \dZln$ we let 
$Q_i=\ThetabarA\left( 0,i \right)$. 
The points $\left\{ P_{e_i}, Q_{f_i} \right\}_{i \in [1..g]}$ form a
symplectic basis of $\Kpol$ for the commutator pairing induced by the theta structure: we
have for $i,j \in [1..g]$  $e_{\pol}(P_{e_i},Q_{f_j})= \delta^i_j$ where
$\delta^i_j$ is the Kronecker symbol.

We have seen (see Section~\ref{subsec@thetastruct}) that the theta
structure $\ThetaA$
induces a section $s=s_{\Kpol}: \Kpol \to \Gpol$ of the canonical projection $\kappa: \Gpol \to \Kpol$. The kernel $K_{\pi}$ of the isogeny $\pi: A_k \to B_k$ is
$\ThetabarA(\dZl)$. Let
$\Ktilde_{\pi} = s(K_{\pi})$ and recall (see
Section~\ref{subsec@isogeniethm}) that $\GpolO =
\Ze(\Ktilde_{\pi})/\Ktilde_{\pi}$. In particular, we have
$\KpolO=\Ze(K_{\pi})/ K_{\pi}$ where
$\Ze(K_{\pi})=\kappa(\Ze(\Ktilde_{\pi}))$ is the orthogonal of
$K_{\pi}$ for the commutator pairing $e_{\pol}$.
Explicitly, we have:
$K_{\pi}=\left\{ Q_i \right\}_{i \in \Zl}$ and
$\Ze(K_{\pi})=\left\{ P_i \right\}_{i \in \Zn} \times \left\{ Q_i
\right\}_{i \in \Zln}$ so that
$\KpolO= \left\{ \pi(P_i) \right\}_{i \in \Zn} \times \pi(\left\{ Q_i
\right\}_{i \in \Zln})$.


%

In the following, if $X_k$ is an abelian variety, we denote by
$\Aut^*(X_k)$ the group of isomorphisms of $X_k$ seen as an
algebraic variety, in particular an element of $\Aut^*(X_k)$ does not
necessarily fix the point $0$ of $X_k$. The action by translation $\rhopol^*: \Kpol \to \Aut^*(A_k)$ induces
an action $\rhopol^*: \Kpol \to \Aut^*(B_k)$ via $\pi$: if $x \in
\Kpol$, the action of $\rhopol^*(x)$ on $B_k$ is the translation by
$\pi(x)$. This action extends the action by translation $\rhopolz^*:
\KpolO \to \Aut^*(B_k)$. We recall that the action of $\Gpol$ on
$V=\Gamma(A_k,\pol)$ is given by $(x,\psi_x).v = \psi_x^{-1}
\tau_x^*(v)$ for $(x,\psi_x) \in \Gpol$ and $v \in V$
from which we derive an action of $\Gpol$ on $\Aff_k(V)$. By restriction, we
obtain an action $\rhopola^*:
\Gpol \to \Aut^*(\Atilde)$. Similarly, we define an action
$\rhopolaz^*: \GpolO \to \Aut^*(\Btilde)$. (Note that $\pitilde$
does not induce an action from $G(\pol)$ on $\Btilde$ via $\rhopola^*$; see
Corollary~\ref{cor@determination}). Still, we would like to be able to
recover $\rhopola^*$ from $\rhopolaz^*$ and the theta
structure $\ThetaB$. First, we have:

\begin{proposition}
  Let $g \in \Ze(\Ktilde_{\pi})$ and note $\overline{g}$ its image in $\Ze(\Ktilde_{\pi}) / \Ktilde_{\pi}$. We have $\rhopolaz^*(\overline{g})=\pitilde \circ \rhopola(g)$.
  \label{prop@compaaction}
\end{proposition}
\begin{proof}
  This is as an immediate consequence of the fact that the two theta structures $\ThetaA$
  and $\ThetaB$ are compatible.
\end{proof}

For $g \in \Gpol$, we can define a mapping $\pitilde_{g}: \Atilde \to
\Btilde$ given on geometric points by $\xtilde \mapsto \pitilde(\rhopola^*(g).\xtilde)$.  If $g \in
\Ze(\Ktilde_{\pi})$; Proposition~\ref{prop@compaaction} then shows
that $\pitilde_{g}=\rhopolaz^*(\overline{g}) \circ \pitilde$, hence
$\pitilde_{g}$ can be recovered from $\pitilde$ and the theta
structure $\ThetaB$. Since $\Ze(\Ktilde_{\pi}) \supset s(\Kpolb)$, we only have to study the mappings
$\pitilde_i=\pitilde_{s(P_i)}$ for $i \in \Zln$. They
are given on geometric points by 
\[ \pitilde_i( (\thetatilde_j(\tilde{x}))_{j \in \Zln}) =
(\thetatilde_{i+\ell.j}(\tilde{x}))_{j \in \Zn}. \] 


\glsadd{glo@Scal}
\glsadd{glo@pitildei}
\begin{corollary}
  \begin{enumerate}
    \item  Let $\mathcal{S}$ be a subset of $\Zln$, such that
    $\mathcal{S}+\Zn=\Zln$. Then $\xtilde \in \Atilde$ is uniquely determined by
  $\{ \pitilde_i(\xtilde) \}_{i \in \mathcal{S}}$.
    \item Let $\ytilde \in \Atilde$ be such that $\pitilde(\ytilde)=\pitilde(\xtilde)$. Then
      there exists $j \in \dZl \subset \dZln$ such that
      $\ytilde=(1,0,j).\xtilde$ and
      \[\pitilde_i(\ytilde)= {e_{\overln}(i,j)} \pitilde_i(\xtilde).\] 
      In particular
      $\pitilde_i(\ytilde)$ and $\pitilde_i(\xtilde)$ differ by an
      $\ell^{th}$-root of unity.
  \end{enumerate}
  \label{cor@determination}
    \end{corollary}

\begin{proof}
~

  \begin{enumerate}
    \item 
      Since $\pitilde_i( (\thetatilde_j(\xtilde))_{j \in \Zln}) =
      (\thetatilde_{i+\ell.j}(\xtilde))_{j \in \Zn}$, from $\left\{
      \pitilde_i(\xtilde) \right\}_{i \in \mathcal{S}}$ one can obtain the
      values $\left\{\thetatilde_j(\xtilde) \right\}_{j \in \mathcal{S}+\Zn}$.
      If $\mathcal{S}+\Zn=\Zln$ this shows that we can recover 
      $\xtilde=(\thetatilde_j(\xtilde))_{j \in \Zln}$.
\item 
  If $\pitilde(\ytilde)=\pitilde(\xtilde)$, then
  $p_{A_k}(\ytilde)-p_{A_k}(\xtilde) \in K_\pi$. So there exists $j \in
  \dZl$ and $\alpha \in \overline{k}^*$ such that
  $\ytilde=(\alpha,0,j).\xtilde$.  Hence $\thetatilde_i(\ytilde)=\alpha
  e_{\overln}(i,j) \thetatilde_i(\xtilde)$. Since
  $\pitilde(\xtilde)=\pitilde(\ytilde)$, $\alpha=1$. Moreover, as  $j \in
  \dZl$, $e_{\overln}(i+k,j)=e_{\overln}(i,j)$ if $k \in \Zn$ so that
  $\pitilde_i(\xtilde)=e_{\overln}(i,j) \pitilde_i(\ytilde)$.
\end{enumerate}
\end{proof}

\begin{example}
  \label{ex@section-torsion}
  \begin{itemize}
    \item 
  If $\ell$ is prime to $n$, the canonical mappings $\Zn \to \Zln$ and $\Zl
  \to \Zln$ induce an isomorphism $\Zn \times \Zl \iso \Zln$, and one
  can take $\mathcal{S}=\Zl$ in Corollary~\ref{cor@determination}.
\item 
  If $\ell$ is not prime to $n$, a possible choice for $\mathcal{S}$
  is 
  \[ \mathcal{S}= \{ \sum_{i \in [1..g]} \lambda_i e_i | \lambda_i \in
  [0..\ell-1] \}. \]
  \end{itemize}
\end{example}

\section{The addition relations}
\label{sec@addition}

In this section  we study the addition relations and introduce the notion
of addition chain on the affine cone of an abelian variety. This chain
addition will be our basic tool for our isogenies computation in
Section~\ref{sec@isogenies} and Vélu's like formula in
Section~\ref{sec@velu}.

In Section~\ref{subsec@evaluation} we introduce the concept of extended
commutator pairing. The importance of this pairing comes from the fact that
the isogenies we compute with our algorithm (see
Section~\ref{subsec@algo2}) correspond to subgroups that are isotropic for
this extended commutator pairing~\cite{0910.4668v1}. We explain in
Section~\ref{sec@pairings} how to use addition chains to compute this
pairing. In Section~\ref{subsec@pseudoadd} we prove in an algebraic
setting the Riemann relations, and we deduce from them the addition
relations. In Section~\ref{subsec@theta_and_addition} we use the results of
Section~\ref{subsec@action} to study the properties of the addition chain.

\subsection{Evaluation of algebraic theta functions at points of $\ell$-torsion}
\label{subsec@evaluation}
  
\glsadd{glo@zerotildeB}
\glsadd{glo@zerotildeA}
For the rest of this article, we suppose that we have fixed a $\zeroB \in
\pB^{-1}(0_{B_k})$. This give us a canonical way to fix an affine lift of
$0_{A_k}$:  we denote
$\zeroA$ the unique point in $p_{A_k}^{-1}(0_{A_k})$ such that $\zeroB =
\pitilde(\zeroA)$. 
We recall that the theta structure
$\ThetaA$ gives a section $s_{K(\pol)}:K(\pol) \to G(\pol)$. 
This means that the map
$x \in K(\pol) \mapsto s_{K(\pol)}(x).\zeroA \in \Atilde$
induces a section $K(\pol) \to \Atilde$ of the map $\pA: \Atilde \to A_k$.
We remark that the choice of $\zeroA \in p_{A_k}^{-1}(0_{A_k}) \subset \Atilde$ is equivalent to the choice of an evaluation isomorphism:
$\epsilon_0 : \pol(0) \simeq k$. 
For any $x
\in K(\pol)$, let $s_{\pol}(x)=(x,\psi_x)$, we define
$\thetatilde_i(x)=\epsilon_0(s_{\pol}(x).\theta_i)= \epsilon_0 \circ \psi_x^{-1} \circ \tau^*_{x}(\theta_i)$.  Then the section $K(\pol) \to \Atilde$ is given by:
\[ s_{K(\pol)}(x).\zeroA = (\thetatilde_i(x) )_{i \in \Zln}. \]

\glsadd{glo@Ptildei}
\glsadd{glo@Rtildei}
Thus we have a canonical way to fix an affine lift for any geometric
point in $K(\pol)$. For $i \in \Zln$, let $\tilde{P}_i=(1,i,0).\zeroA$, and
for $j \in \dZln$, let $\tilde{Q}_j=(1,0,j).\zeroA$. We also put
$\Rtilde_i= \pitilde(\Ptilde_i)=\pitilde_i(\zeroA)$, and
$R_i=\pB(\Rtilde_i)$. We remark that $\{R_i\}_{i \in \Zl}$ is the kernel
$K_{\pidual}$ of $\pidual$ which explains the primordial role the points
$\Rtilde_i$ will play for the rest of this paper.

\glsadd{glo@Bell}
\glsadd{glo@elliso}
More generally, we can define an affine lift
for any point of $B_k[\ell]$ by considering the isogeny given by $[\ell]$
rather than by $\pi$:
let $\bpol = [\ell]^* \ppol$ on $B_k$. As $\ppol$ is symmetric, we
have that $\bpol \simeq \ppol^{\ell^2}$~\cite{MR0282985} and so
$K(\bpol)$, the kernel of $\bpol$ is isomorphic to $K(\overline{\ell^2
n})$. Let $\Thetall$ be a theta structure on $(B_k, \bpol)$ compatible with
the theta structure $\ThetaB$ on $(B_k, \pol_0)$.
We note $\Bell$ the affine cone of $(B_k, \bpol)$, and
$\ellisotilde$ the morphism $\Bell \to \Btilde$ induced by $\elliso$.

We recall that there is a natural action of $G(\bpol)$ on $H^0(\bpol)$
which can be transported via $\Thetall$ to an action of $\Hll$ on
$H^0(\bpol)$. We note $\zeroBl$ the affine lift of the theta null point of
$\Bell$ such that $\ellisotilde \zeroBl = \zeroB$. We can then generalize
the definition of the $\Rtilde_i$ by looking at the points: \[ \{
\ellisotilde (1, i, j).\zeroBl | i,j \in \Zll \times \dZll \}. \] Let $H=\{
(1, i, j) | i, j \in \Zl \}$. $H$ is a commutative subgroup of $\Hll$ such
that $\ellisotilde \xtilde_1 = \ellisotilde \xtilde_2$ if and only if
$x_1=h.x_2$ where $h \in H$. We see that the section of a point in
$B_k[\ell]$ is only defined up to an $\ell^{th}$-root of unity. (This is
also the case for the lift $\Rtilde_i$ of $R_i$: if we change the theta
structure on $A_k$, it changes the $\Rtilde_i$ by an $\ell^{th}$-root of
unity. See also~Corollary~\ref{cor@determination} and
Example~\ref{ex@ltorsion}). The geometric meaning of these affine lifts is
explained in Section~\ref{subsec@trueltorsion}.

\glsadd{glo@el}
The polarization $\bpol$ induces a commutator pairing
$e_{\bpol}$ (\cite{MR34:4269}) on $K(\bpol)$ and as $\bpol$ descends to $\ppol$ via the
isogeny $[\ell]$, we know that $e_{\bpol}$ is trivial on $B_k[\ell]$.
For $x_1,x_2 \in B_k[\ell]$, let $x'_1, x'_2 \in B_k[\ell^2]$ be such
that $\ell.x'_i =x_i$ for $i=1,2$. We remark that $x'_1$ and $x'_2$
are defined up to an element of $B_k[\ell]$. As a consequence,
$e_{\bpol}(x'_1,x_2) =e_{\bpol}(x_1,x'_2)=e_{\bpol}(x'_1,x'_2)^{\ell}$, does not depend on the
choice of $x'_1$ and $x'_2$ and if we put
$e'_\ell(x_1,x_2)=e_{\bpol}(x'_1,x_2)$, we obtain a well defined bilinear
application $e'_\ell : B_k[\ell n] \times B_k[\ell n] \rightarrow
\overline{k}$. We call $e'_\ell$ the extended commutator pairing. It extends the
commutator pairing $e_{\pol_0}$ since if $x,y \in B_k[n]$ we have
$e'_\ell(x,y)=e_{\pol_0}(x,y)^{\ell}$. When we speak of isotropic points in
$B_k[\ell]$, we always refer to isotropic points with respect to $e'_\ell$.
We quote the following important result from~\cite{0910.4668v1}: the
modular points $\phi_1^{-1}(b_i)_{i \in \Zn}$ that we compute in
Section~\ref{subsec@algo2} correspond to isogenies whose kernel are
isotropic for $e'_\ell$.


\subsection{The general Riemann relations}
\label{subsec@pseudoadd}

The Riemann relations~\eqref{eq@equations_Mn} for $\Mln$ and the Riemann
equations~\eqref{eq@equations_Ak} for $A_k$ are all particular case of more
general Riemann relations, which we will use to get the addition relations
on $A_k$. An analytic proof of these relations can be found
in~\cite{MR85h:14026}.

\begin{theorem}[Generalized Riemann Relations]
  \label{th_riemann_rel}
  Let $x_1,y_1,u_1,v_1,z \in A_k$ be such that $x_1+y_1+u_1+v_1=2z$. 
  Let $x_2 =z-x_1$, $y_2=z-y_1$, $u_2=z-u_1$ and $v_2=z-v_1$.
  Then there exist 
  $\tilde{x}_1 \in \pA^{-1}(x_1)$,
  $\tilde{y}_1 \in \pA^{-1}(y_1)$,
  $\tilde{u}_1 \in \pA^{-1}(u_1)$,
  $\tilde{v}_1 \in \pA^{-1}(v_1)$,
  $\tilde{x}_2 \in \pA^{-1}(x_2)$,
  $\tilde{y}_2 \in \pA^{-1}(y_2)$,
  $\tilde{u}_2 \in \pA^{-1}(u_2)$,
  $\tilde{v}_2 \in \pA^{-1}(v_2)$ that satisfy the following relations:
  for any $i,j,k,l,m \in \Zln$ such that $i+j+k+l=2m$, let $i'=m-i$,
  $j'=m-j$, $k'=m-k$ and $l'=m-l$, then for all $\chi \in \dZtwo$, we have
\begin{multline}
  \big(\sum_{t \in \Ztwo} \chi(t) \vartheta_{i+t}(\tilde{x}_1) \vartheta_{j+t}(\tilde{y}_1)\big).\big(\sum_{t \in \Ztwo} \chi(t)
  \theta_{k+t}(\tilde{u}_1) \theta_{l+t}(\tilde{v}_1)\big)= \\
  \big(\sum_{t \in \Ztwo} \chi(t) \vartheta_{i'+t}(\tilde{x}_2) \vartheta_{j'+t}(\tilde{y}_2)\big).\big(\sum_{t \in \Ztwo} \chi(t)
  \theta_{k'+t}(\tilde{u}_2) \theta_{l'+t}(\tilde{v}_2)\big). 
\label{eq_riemann_rel}
\end{multline}
\end{theorem}

\begin{proof}
  If $x=y=u=v=0_A$, the preceding result gives the algebraic Riemann
  relations, a proof of these relations can be found in~\cite{MR34:4269}.
  We just need to adapt the proof of Mumford for the general case.

  Let $p_1$ and $p_2$ be the first and second projections from $A_k
  \times A_k$ to $A_k$.
  Let $\Mpol=\pullb{p_1}(\pol) \prodtens \pullb{p_2}(\pol)$. The theta
  structure $\ThetaA$ induces a theta structure $\Theta_{A_k \times A_k}$ such
  that for $(i,j) \in \Zln \times \Zln$ we have $\theta^{ \Theta_{A\times
  A}}_{i,j}=\theta^{\ThetaA}_i \prodtens \theta^{\ThetaA}_j$. (see \cite[Lemma~1 p.
  323]{MR34:4269}.) Consider the isogeny $\xi: A_k \times A_k \to A_k \times
  A_k, (x,y) \mapsto (x+y,x-y)$.  We have $\pullb{\xi}(\Mpol)=\Mpol^2$.
  Since $\ThetaA$ is a symmetric theta structure $\ThetaA$ it induces a
  theta structure on $\pol^2$ and on $\Mpol^2$.
  One can check that this theta structure is compatible with the
  isogeny $\xi$.  Applying the isogeny theorem
  (see~\cite[p324]{MR34:4269}), we obtain that there exists $\lambda
  \in \overline{k}^*$ such that for all $i,j \in \Zln$ and $x,y \in
  A_k(\overline{k})$:
  \begin{equation}
  (\theta^{\pol}_i \tens \theta^{\pol}_j) (\xi(x,y)) = \lambda 
  \sum_{\substack{u,v \in \Zstruct{\overline{2ln}} \\ u+v=i \\ u-v=j}}(\theta_u^{\pol^2} \tens
  \theta_v^{\pol^2})(x,y) 
    \label{th_add_eq1}
  \end{equation}

In the preceding equation, in order to evaluate the sections of
$\Mpol$ or $\Mpol^2$ at a geometric point of $x\in A_k \times A_k$ we just
choose any isomorphism between $\Mpol_x$ or $\Mpol^2_x$ and
$\mathscr{O}_{A_k \times A_k,x}$ where $\mathscr{O}_{A_k \times A_k}$
is the structural sheaf of $A_k \times A_k$. To simplify the
notations, we suppose in the following that $\lambda =1$.
  
  Using equation~\eqref{th_add_eq1} we compute for all $i,j \in
  \Zstruct{\overline{2\ell n}}$ which are congruent modulo $\Zln$
  and $x,y \in A_k(\overline{k})$:
  \begin{align*}
    \label{th_add_eq2}
    \sum_{t \in \Ztwo} \chi(t) \theta^{\pol}_{i+j+t}(x+y) \theta^{\pol}_{i-j+t}(x-y) 
     &= 
     \sum_{\substack{t \in \Ztwo \\  u,v \in \Zstruct{\overline{2ln}} \\  u+v=i+j+t \\  u-v=i-j+t}}\chi(t)\theta_u^{\pol^2}(x) \theta_v^{\pol^2}(y) \\
     &=
     \sum_{t_1,t_2 \in \Ztwo}\chi(t_1+t_2)\theta_{i+t_1}^{\pol^2}(x) \theta_{j+t_2}^{\pol^2}(y) \\
     &= 
     \left(\sum_{t \in \Ztwo}\chi(t)\theta_{i+t}^{\pol^2}(x)\right) \cdot
     \left(\sum_{t \in \Ztwo}\chi(t)\theta_{j+t}^{\pol^2}(y)\right)
  \end{align*}
So we have:
  \begin{multline}
\big(\sum_{t \in \Ztwo} \chi(t) \theta^\pol_{i+j+t}(x+y) \theta^\pol_{i-j+t}(x-y)\big).
\big(\sum_{t \in \Ztwo} \chi(t) \theta^\pol_{k+l+t}(u+v) \theta^\pol_{k-l+t}(u-v)\big)= \\
     \big(\sum_{t \in \Ztwo}\chi(t)\theta_{i+t}^{\pol^2}(x)\big) \cdot
     \big(\sum_{t \in \Ztwo}\chi(t)\theta_{j+t}^{\pol^2}(y)\big) \cdot
     \big(\sum_{t \in \Ztwo}\chi(t)\theta_{k+t}^{\pol^2}(u)\big) \cdot
     \big(\sum_{t \in \Ztwo}\chi(t)\theta_{l+t}^{\pol^2}(v)\big) = \\
\big(\sum_{t \in \Ztwo} \chi(t) \theta^\pol_{i+l+t}(x+v) \theta^\pol_{i-l+t}(x-v)\big).
\big(\sum_{t \in \Ztwo} \chi(t) \theta^\pol_{k+j+t}(u+y) \theta^\pol_{k-j+t}(u-y)\big)
    \label{th_add_eq3}
  \end{multline}
  
  Now if we let $x=x_0+y_0$, $y=x_0-y_0$, $u=u_0+v_0$ and $v=u_0-v_0$, we
  have $x+y+u+v=2 (x_0+u_0)$ so we can choose $z=x_0+u_0$, so that
  $z-x=u_0-y_0$, $z-y=u_0+y_0$, $z-u=x_0-v_0$, $z-v=x_0+v_0$.
  By doing the same variable change for $i,j,k,l$ we see that
  the theorem is just a restatement of Equation~\eqref{th_add_eq3}.
  (see~\cite[p334]{MR34:4269}).
\end{proof}

From the generalized Riemann relations it is possible to derive addition
relations. First remark that since the theta structure $\ThetaA$ is
symmetric there exists a constant $\lambda \in \overline{k}^*$ such
that for all $i \in \Zln$, $\theta_i(-x) = \lambda.\theta_{-i}(x)$. 
In
particular if $\xtilde \in \Atilde$, we put $- \xtilde
=(\thetatilde_{-i}(\xtilde))_{i \in \Zln}$.
\glsadd{glo@chainadd}
\begin{theorem}[Addition Formulas]
  \label{th@addition}
  Let $x, y \in A_k(\overline{k})$ and suppose that we are given $\xtilde \in
  \pA^{-1}(x)$, $\ytilde \in \pA^{-1}(y)$, $\tilde{x-y} \in
  \pA^{-1}(x-y)$, then there is a unique point $\tilde{x+y} \in
  \tilde{A}_k(\overline{k})$ verifying for $i,j,k,l \in
  \Zln$ 
\begin{multline}
  \big(\sum_{t \in \Ztwo} \chi(t) \theta_{i+t}(\tilde{x+y}) \theta_{j+t}(\tilde{x-y})\big).\big(\sum_{t \in \Ztwo} \chi(t)
\theta_{k+t}(\zeroA) \theta_{l+t}(\zeroA)\big)= \\
\big(\sum_{t \in \Ztwo} \chi(t) \vartheta_{-i'+t}(\tilde{y}) \vartheta_{j'+t}(\tilde{y})\big).\big(\sum_{t \in \Ztwo} \chi(t)
\theta_{k'+t}(\tilde{x}) \theta_{l'+t}(\tilde{x})\big), 
\label{eq@addition}
\end{multline}
and we have $\pA(\tilde{x+y})=x+y$.

  Thus the addition law on $A_k$ extends to a pseudo addition law on
  $\Atilde$; we call it an addition chain and we note
  $\tilde{x+y}=\chaineadd(\tilde{x},\tilde{y},\tilde{x-y})$.
\end{theorem}

\begin{proof}
  We apply the Riemann relations~\eqref{eq_riemann_rel} to
  $x+y,x-y,0_A,0_A$. We have $2x=(x+y)+(x-y)+0_A+0_A$, $-y=x-(x+y)$,
  $y=x-(x-y)$, $x=x-0_A$, $x=x-0_A$ so Theorem~\ref{th_riemann_rel} shows
  that there exist a point $\tilde{x+y}\in \tilde{A}_k(\overline{k})$ satisfying the addition
  relations~\eqref{eq@addition}.

  It remains to show that this point is unique. But first we reformulate
  the addition formulas (see~\cite[p334]{MR34:4269}). Let $H=\Zln \times
  \dZtwo$, and for $(i,\chi) \in H$ define 
  \[\utilde_{i,\chi}(\xtilde)=\sum_{t \in \Ztwo} \chi(t)
  \thetatilde_{i+t}(\xtilde). \]
  Then we have for all $i,j,k,l,m \in H$ such that $2m=i+j+k+l$
  \begin{multline}
    \utilde_i(\tilde{x+y})\utilde_j(\tilde{x-y})\utilde_k(\zeroA)\utilde_l(\zeroA)= \\
    \frac{1}{2^{2g}} \sum_{\xi \in H, 2 \xi=\in \Ztwo \times {0}}
    (m_2+\xi_2)(2\xi_1) \utilde_{i-m+\xi}(\ytilde)
    \utilde_{m-j+\xi}(\ytilde) \utilde_{m-k+\xi}(\xtilde)
    \utilde_{m-l+\xi}(\xtilde)
  \label{eq@addition_trans}
  \end{multline}

  It is easy to see that $(\thetatilde_i(\tilde{x}))_{i \in \Zln}$, is determined by
  $(\utilde_i(\tilde{x}))_{i \in H}$. That means there is a $j \in H$ such that
  $\utilde_j(\tilde{x-y}) \ne 0$ otherwise we would have
  $\thetatilde_i(\tilde{x-y})=0$ for all $i
  \in \Zln$. Now for all $i \in H$, we need to find $k,l \in H$ such that
  $i+j+k+l=2m$ and $\utilde_k(\tilde{0}) \ne 0$, $\utilde_l(\tilde{0}) \ne 0$. In that case
  equation~\eqref{eq@addition_trans} allows us to compute
  $\utilde_i(\tilde{x+y})\utilde_j(\tilde{x-y})\utilde_k(0)\utilde_l(0)$
  from $(\utilde_i(\tilde{x}))_{i \in H}$ and
  $(\utilde_i(\tilde{y}))_{i \in H}$, so we can obtain
  $\utilde_i(\tilde{x+y})$.

  In \cite[p. 339]{MR34:4269}, Mumford prove that for any $i \in H$,
  there is an $\alpha \in 2 \Zln$ such that $\utilde_{i+(\alpha,0)}(0) \neq 0$.  
  So we can choose $k=i$, $l=j$, we have $i+j+k+l=2m$, and if necessary we add
  an element of $2 \Zln$ to $k$ and $l$.
\end{proof}

\begin{algorithm}[Addition chain]
  \begin{description}
    \item[Input]
  $\xtilde, \ytilde$ and $\tilde{x-y}$ such that
  $\pA(\xtilde)-\pA(\ytilde)=\pA(\tilde{x-y})$.
\item[Output] 
  $\tilde{x+y}=\chaineadd(\xtilde,\ytilde,\tilde{x-y})$.
\item[Step 1] 
  For all $i \in \Zln$, $\chi \in \dZtwo$ and $X \in
  \{\tilde{x+y},\tilde{x}, \tilde{y}, \zeroA\}$ compute
  \[\utilde_{i,\chi}(X)=\sum_{t \in \Ztwo} \chi(t) \thetatilde_{i+t}(X). \]
\item[Step 2] 
  For all $i \in \Zln$ choose $j,k,l \in \Zln$ such that $i+j+k+l=2m$, 
  $\utilde_j(\tilde{x-y}) \ne 0$, $\utilde_k(\zeroA) \ne 0$, 
  $\utilde_l(\zeroA) \ne 0$ and compute
  \begin{multline}
    \utilde_i(\tilde{x+y})=\frac{1}{2^{2g}\utilde_j(\tilde{x-y})\utilde_k(\zeroA)\utilde_l(\zeroA)} \\
    \sum_{\xi \in H, 2 \xi=\in \Ztwo \times {0}}
    (m_2+\xi_2)(2\xi_1) \utilde_{i-m+\xi}(\ytilde)
    \utilde_{m-j+\xi}(\ytilde) \utilde_{m-k+\xi}(\xtilde)
    \utilde_{m-l+\xi}(\xtilde).
    \label{eq@addition_class}
  \end{multline}
\item[Step 3] For all $i \in \Zln$ output
  \[ \thetatilde_{i}(\tilde{x+y})=\frac{1}{2^g} \sum_{\xi \in \dZtwo}
  \utilde_{i, \chi}(\tilde{x+y}). \]
  \end{description}
\end{algorithm}

\begin{complexity}
  We use a linear transformation between the $\thetatilde$ coordinates and the
  $\utilde$ coordinates in the Steps~1 and~3 because we have seen in the
  proof of Theorem~\ref{th@addition} that the latter are more suited for
  additions.

  As $\utilde_{i+t,\chi}=\chi(t) \utilde_{i,\chi}$
  we only need to consider $(\ell n)^g$ coordinates and the
  linear transformation between $\utilde$ and $\thetatilde$ 
  can be computed at the cost of $(2 n
  \ell)^g$ additions in $k$.
  We also have $\utilde_{i, \chi}(-\xtilde)=\utilde_{-i, \chi}(\xtilde)$.

  Using the fact that for $t \in \Ztwo$ the right hand terms
  of~\eqref{eq@addition_class} corresponding to $\xi=(\xi_1+t,\xi_2)$ and
  to $\xi=(\xi_1,\xi_2)$ are the same up to a sign, one can compute the
  left hand side of~\eqref{eq@addition_class} with $4\cdot 4^g$
  multiplications and $4^g$ additions in $k$. In total one can compute an
  addition chain in $4.(4\ell n)^g$ multiplications, $(4\ell n)^g$
  additions and $(\ell n)^g$ divisions in $k$. We remark that in order to
  compute several additions using the same point, there is no need to
  convert back to the $\thetatilde$ at each step so we only need to perform
  Step~2.

  The addition chain formula is a basic step for our isogenies
  computations, and in the sequel we consider it as our basic unit for
  the complexity
  analysis. In some cases it is possible to greatly speed up this computation. See
  for instance \cite{gaudry2007fast} which uses the duplication formula
  between theta functions to speed up the addition chain of level two (in
  genus $1$ and $2$). See
  also Section~\ref{subsec@compression} where it is explained how to
  use isogenies to compute the addition chain for a general level by using
  only addition chains of level two, so that we can
  use the speed up of \cite{gaudry2007fast} in every level.
\end{complexity}

\begin{remark}
  The addition formulas can also be used to compute the usual addition law
  in $A_k$  by choosing $j=0$ in
  Equation~\eqref{eq@addition_class} for every $i$.
 
  It is also possible to use directly  the $\theta$ coordinates as
  follows:
  if $x, y \in A_k(\overline{k})$ and $i \in \Zln$, for every $\chi \in \dZtwo$, one can choose $k,l \in \Zln$  such that $i+k+l$ is divisible by $2$ and

\[ \big(\sum_{t \in \Ztwo} \chi(t) \theta_{k+t}(0) \theta_{l+t}(0)\big)
\neq 0. \] 
So one can use~\eqref{eq@addition} to compute 
 \[ a_{i,\chi}:= \big(\sum_{t \in \Ztwo} \chi(t) \vartheta_{i+t}(P+Q)
 \vartheta_{0+t}(P-Q)\big), \]
 and recover $\theta_{i}(P+Q)$ by inversing a matrix. For
 instance we have $\displaystyle\theta_i(x)=\frac{1}{2^g} \sum_{\chi \in
 \dZtwo} a_{\chi,i}$
\end{remark}

The addition chain law on $\Atilde$ induces a multiplication by a
scalar law which reduces via $p_{A_k}$ to the
usual multiplication by a scalar law deduced from the group structure
of $A_k$. Let $\xtilde, \ytilde \in \Atilde$ and
$\tilde{x+y} \in p_{A_k}^{-1}(x+y)$, then we can compute
$\tilde{2x+y}:=\chaineadd(\xytilde, \xtilde, \ytilde)$. More generally 
there is a recursive algorithm to compute for every $m \geq
2$:
\[ \tilde{mx+y}:=\chaineadd(\tilde{(m-1)x+y}, \xtilde, \tilde{(m-2)x+y}) \] 
We put $\chaineaddmult(n,\xytilde, \xtilde,
\ytilde):= \tilde{mx+y}$ and define
$\chainemult(m,\xtilde):=\chaineaddmult(m, \xtilde, \xtilde,\zeroA)$. 
We have $p_{A_k}(\chainemult(m,\xtilde))=m. p_{A_k}(\xtilde)$.
We call $\chaineaddmult$ a multiplication chain.
\glsadd{glo@chainmult}

\begin{algorithm}[Multiplication chain]
  \begin{description}
    \item[Input] $m \in \N$, $\xytilde, \xtilde, \ytilde \in \Atilde$.
    \item[Output] $\chaineaddmult(m,\xytilde,\xtilde,\ytilde)$.
    \item[Step 1] Compute the binary decomposition of 
       $m:=\sum_{i=0}^I b_i 2^i$.
       Set $m':=0$, 
       $\mathtt{xy}_0:=\ytilde$,
       $\mathtt{xy}_{-1}:=\chaineadd(\ytilde,-\xtilde,\xytilde)$,
       $\mathtt{x}_0:=\zeroA$ and $\mathtt{x}_1:=\xtilde$.
      \item[Step 2] For i in $[I .. 0]$ do\\
        If $b_i=0$ then compute
        \begin{gather*}
          \mathtt{x}_{2m'}:=\chaineadd(\mathtt{x}_{m'},\mathtt{x}_{m'},\mathtt{x}_0) \\
        \mathtt{x}_{2m'+1}:=\chaineadd(\mathtt{x}_{m'+1},\mathtt{x}_{m'},\mathtt{x}_1) \\
        \mathtt{xy}_{2m'}:=\chaineadd(\mathtt{xy}_{m'},\mathtt{x}_{m'},\mathtt{xy}_0)\\
        m':=2m'.
      \end{gather*}
      Else compute
        \begin{gather*}
        \mathtt{x}_{2m'+1}:=\chaineadd(\mathtt{x}_{m'+1},\mathtt{x}_{m'},\mathtt{x}_1) \\
          \mathtt{x}_{2m'+2}:=\chaineadd(\mathtt{x}_{m'+1},\mathtt{x}_{m'+1},\mathtt{x}_0) \\
          \mathtt{xy}_{2m'+1}:=\chaineadd(\mathtt{xy}_{m'},\mathtt{x}_{m'},\mathtt{xy}_{-1})\\
        m':=2m'+1.
      \end{gather*}
    \item[Step 3] Output $\mathtt{xy}_m$.
  \end{description}
\end{algorithm}
\begin{algocomplexity}
  In Corollary~\ref{cor@fastmul} we show that multiplication chains are
  associative, so we can use a Lucas sequence to compute them. In order to
  do as few division as possible, we use a Montgomery ladder for our Lucas
  sequence, hence the algorithm.

  We see that a multiplication chain requires $O(\log(m))$ addition chains.
\end{algocomplexity}

\begin{lemma}
  For $\lambda_x, \lambda_y, \lambda_{x-y} \in \overline{k}^*$ and
  $\tilde{x},\tilde{y} \in A_k(\overline{k})$, we have:
  \begin{gather}
    \label{eq@homogen1}
    \chaineadd(\lambda_x \xtilde, \lambda_y \ytilde, \lambda_{x-y}
    \xymtilde) = \frac{\lambda_x^2 \lambda_y^2}{\lambda_{x-y}}
    \chaineadd(\xtilde, \ytilde, \xymtilde), \\
    \label{eq@homogen2}
    \chaineaddmult(n, \lambda_{x+y} \xytilde, \lambda_x \xtilde, \lambda_y
    \ytilde) = \frac{\lambda_x^{n(n-1)}
    \lambda_{x+y}^{n}}{\lambda_{y}^{n-1}}
    \chaineaddmult(n, \xytilde,  \xtilde, \ytilde), \\
    \chainemult(n, \lambda_x \xtilde) = 
    \lambda_x^{n^2} \chainemult(n, \xtilde). 
    \label{eq@homogen3}
  \end{gather}
  \label{lem@eqhomogen}
\end{lemma}

\begin{proof}
  Formula \eqref{eq@homogen1} is an immediate consequence of the addition
  formulas~\eqref{eq@addition}. The rest of the lemma follows by an
  easy recursion.
\end{proof}

\subsubsection{The case $n=2$}
\label{subsubsec@addition_level2}
Let $x$ be the generic point of $A_k$. Then the Riemann equations~\eqref{eq@riemann_eq} come
from the following addition formula:
\begin{equation}
 x=\chaineadd(x,0_A,x). 
  \label{eq@riemann_eq}
\end{equation}


We suppose here that $n=2$ and $\ell=1$. We have for all $i \in \Ztwo$, $(-1)^*\theta_i=\theta_i$,
where $(-1)$ is the inverse automorphism on $A_k$. As a consequence,
$\pol$ gives
an embedding of the Kummer variety $K_A= A/\pm 1$.
The equations~\eqref{eq@riemann_eq} are trivial, so that Riemann equations
does not give the projective equations of this embedding (except when
$g=1$). Nonetheless, one can recover these equations by
considering more general addition relations. 

Let $p: A_k \to K_A$ be the natural projection. If $x,y \in K_A$, let $x_0
\in p^{-1}(x)$ and $y_0=p^{-1}(y)$, we have $p(p^{-1}(x)+p^{-1}(y))=p(\pm
x_0 + \pm y_0)=\{p(x_0+y_0), p(x_0-y_0)\}$. As a consequence, there is
no properly defined addition
law on $K_A$: from $\pm x \in K_A$ and $\pm y \in K_A$, we may compute $\pm
x \pm y$ which give two points on $K_A$. However, if we are also given $\pm
(x-y) \in K_A$, then we can identify $\pm (x+y) \in \{ \pm x \pm y \}$.
Thus the addition chain law from~Theorem~\ref{th@addition} extends to a pseudo
addition on the Kummer variety. We remark that our addition chain is a
generalization of the pseudo addition law on the Kummer variety.

Let $x,y \in K_A$. 
To compute $\pm x \pm y$ without $\pm(x-y)$ we can proceed as follows:
let $X=(X_i)_{i \in \Ztwo}$, $Y=(Y_i)_{i \in \Ztwo}$ be
the two projections of the generic point on $K_A \times K_A$. Then the
addition relations $\chaineadd(X,x,y,Y)$ describe a system of degree~$2$ in
$K_A \times K_A$, whose solutions are $(\pm (x+y), \pm (x-y))$ and $(\pm
(x-y), \pm (x+y))$. From this system, it is easy to recover the points $\{
\pm (x+y), \pm (x-y) \}$, but this involves a square root in $k$. (The
preceding claims will be proved on an upcoming paper). Coming back to isogenies
computations, it means that when working with $n=2$, we have to avoid
computing normal additions, since they require
a square root and are much slower than addition chains. 

We make a last remark concerning additions on the Kummer variety.
Suppose that we are given $x,y,z \in K_A$, together with $\pm (x+y)$, $\pm
(y+z)$. We want to find $\pm (x+z)$. Using the addition relations,
we can compute $S=\{ \pm(x+z), \pm(x-z) \}$. Let $A \in S$, then the
solutions of the addition relations $\chaineadd(X,x+y,A,Y)$ are $\{\pm (2x+y+z), \pm (y-z) \}$ if $A= \pm(x+z)$, and $\{ \pm (2x+y-z), \pm (y+z) \}$ if 
$A=\pm (x-z)$. This allows to find $\pm (y+z) \in S$ if  $2x \ne
0, 2y \ne 0, 2z \ne 0, 2(x+y+z) \ne 0$. We call this the compatible
addition relation, and we can use this to compute $\pm(x+z)$
directly without taking a square root (by computing the gcd between the
two systems of degree two given by the addition relations).

\subsection{Theta group and addition relations}
\label{subsec@theta_and_addition}
In this Section, we study the action of the theta group on the addition
relations. We use this action to find the addition relations linking
the coordinates of the points 
$\{ \tilde{R}_i \}_{i \in \Zln}$. By considering different isogenies 
$\pi: A_k \to B_k$, we can then understand the addition chains between
any isotropic subgroup of $B_k[\ell]$ (see~Section~\ref{subsec@evaluation}). In particular we exploit this to show that we can compute the chain
multiplication by $\ell$ in $O(\log(\ell))$ addition chains.

\begin{lemma}
  \label{lem@compa_isogenies}
  Suppose that $\tilde{x}_1, \tilde{y}_1, \tilde{u}_1, \tilde{v}_1, \tilde{x}_2,
  \tilde{y}_2, \tilde{u}_2, \tilde{v}_2 \in \Atilde$ satisfy the general Riemann
  relations~\eqref{eq_riemann_rel}. 
  \begin{itemize}
    \item 
  For every $g \in G(\pol)$, 
  $g.\tilde{x}_1, g.\tilde{y}_1, g.\tilde{u}_1, g.\tilde{v}_1, g.\tilde{x}_2,
  g.\tilde{y}_2, g.\tilde{u}_2, g.\tilde{v}_2$ also satisfy the Riemann
  relations.
    \item 
      For every isogeny $\pi: (A, \pol, \ThetaA) \to (B, \pol_0,
      \ThetaB)$ such that $\ThetaB$ is $\pi$-compatible of type~$1$ with $\ThetaA$,
      then
  $\pitilde(\tilde{x}_1), \pitilde(\tilde{y}_1), \pitilde(\tilde{u}_1), \pitilde(\tilde{v}_1), \pitilde(\tilde{x}_2), \pitilde(\tilde{y}_2), \pitilde(\tilde{u}_2), \pitilde(\tilde{v}_2) \in \Btilde$ 
  also satisfy the Riemann relations.
  \end{itemize}
\end{lemma}
\begin{proof}
  This is an immediate computation.
\end{proof}

\begin{lemma}
  \label{lem@oppose_point}
  Let $(\alpha,i,j) \in \Hln$. Let $\tilde{x}=(\alpha,i,j).\zeroA \in \Atilde$.
  Then we have $-\tilde{x}=(\alpha,-i,-j).\tilde{0}_{A_k}$.

  More generally, if $\tilde{x} \in \Atilde$, then
  $-(\alpha,i,j).\tilde{x}=(\alpha,-i,-j).-\tilde{x}$, and we have
  $\pitilde(-x)=-\pitilde(x)$.
\end{lemma}
\begin{proof}
  If $\tilde{x}=(x_i)_{i \in \Zln}$, we recall that we have defined
  $-\tilde{x}=(x_{-i})_{i \in \Zln}$.
  Let $\zeroA=(a_i)_{i \in \Zln}$, if $u \in \Zln$ we
  have by~\eqref{eq@actiontheta}:
  $x_u= ((\alpha,i,j).\zeroA)_u= \alpha j(-u-i) a_{u+i}$,
  $((\alpha,-i,-j).\zeroA)_{-u}= \alpha (-j)(u+i) a_{-u-i} = a_{u+i}=x_u$.

  The generalization and the rest of the lemma is trivial.
\end{proof}

An interesting property of the addition formulas, is that they are
compatible with the action $s_{\Kpola}:K_1(\pol) \to \tilde{A}_k$:

\begin{proposition}[Compatibility of the pseudo-addition law]
  For $\tilde{x}, \tilde{y}, \tilde{x-y} \in \Atilde$, and $i,j \in
  \Zln$, we have:
  \begin{equation}
  (1,i+j,0).\chaineadd(\tilde{x},\tilde{y}, \tilde{x-y}) = 
      \chaineadd( (1,i,0).\tilde{x},(1,j,0).\tilde{y}, (1,i-j,0).\tilde{x-y})
    \label{eq@compaadd1}
  \end{equation}
  In particular if we set $\Ptilde_i= (1,i,0).\zeroA$ we have:
  \[ \Ptilde_{i+j}= \chaineadd(\Ptilde_i, \Ptilde_j, \Ptilde_{i-j}) \]
  \label{prop@compaadd}
\end{proposition}
\begin{proof}
  Let $\tilde{x+y}=\chaineadd(\tilde{x},\tilde{y}, \tilde{x-y})$.
  By Theorem~\ref{th@addition}, we have for every $a,b,c,d,e \in \Zln$ such
  that $a+b+c+d=2e$:

  \begin{multline}
    \big(\sum_{t \in \Ztwo} \chi(t) \vartheta_{a+t}(\tilde{x+y})
    \vartheta_{b+t}(\tilde{x-y})\big).\big(\sum_{t \in \Ztwo} \chi(t)
    \theta_{c+t}(\tilde{0}) \theta_{d+t}(\tilde{0})\big)= \\
    \big(\sum_{t \in \Ztwo} \chi(t) \vartheta_{-e+a+t}(\tilde{y})
    \vartheta_{e-b+t}(\tilde{y})\big).\big(\sum_{t \in \Ztwo} \chi(t)
    \theta_{e-c+t}(\tilde{x}) \theta_{e-d+t}(\tilde{x})\big). 
  \label{eq_addition_ex1}
  \end{multline}

  Applying~\eqref{eq_addition_ex1} to $a'=a+i+j, b'=b+i-j,
  c'=c, d'=d, e'=e+i$, it comes:
  \begin{multline}
    \big(\sum_{t \in \Ztwo} \chi(t) \vartheta_{i+j+a+t}(\tilde{x+y})
    \vartheta_{b+i-j+t}(\tilde{x-y})\big).\big(\sum_{t \in \Ztwo} \chi(t)
    \theta_{c+t}(\tilde{0}) \theta_{d+t}(\tilde{0})\big)= \\
    \big(\sum_{t \in \Ztwo} \chi(t) \vartheta_{-j-e+a+t}(\tilde{y})
    \vartheta_{j+e-b}(\tilde{y})\big).\big(\sum_{t \in \Ztwo} \chi(t)
    \theta_{i+e-c+t}(\tilde{x}) \theta_{i+e-d+t}(\tilde{x})\big). 
  \label{eq_addition_ex2}
  \end{multline}

  Thus $(1,i+j,0).\tilde{x+y}$, $(1,i,0).\tilde{x}$,
  $(1,j,0).\tilde{y}$ and $(1,i-j,0).\tilde{x-y}$
  satisfy the additions relations.
\end{proof}

By applying $\pitilde$, we obtain the following corollary:
\begin{corollary}
  \[\pitilde_{i+j}( 
  \chaineadd(\xtilde,\ytilde,\xymtilde))=
  \chaineadd(\pitilde_i(\xtilde),\pitilde_j(\ytilde),\pitilde_{i-j}(\xymtilde))
  \]
  \label{cor@compaadd}
\ 
\end{corollary}
\begin{proof}
Remember that by definition
  $\pitilde_i(\xtilde)=\pitilde( (1,i,0).\xtilde)$.
  The lemma is then a trivial consequence of Proposition~\ref{prop@compaadd} and
  Lemma~\ref{lem@compa_isogenies}. \end{proof}

We discuss two consequences of the preceding corollary. The first is that
we can recover the point $\tilde{x}=(\theta_i(\tilde{x}))_{i \in \Zln}$ from only a subset
of its coordinates $\{\theta_i(\tilde{x})\}_{i \in \Zln}$ (see
Section~\ref{subsec@compression}).

The second application is the computation of the dual isogeny $\pidual$
(see Section~\ref{subsec@isogeny}).

But first we remark that by setting $\xtilde=\ytilde=\zeroA$ in
Corollary~\ref{cor@compaadd}, we find
\[ \Rtilde_{i+j}=\chaineadd(\Rtilde_i,\Rtilde_j,\Rtilde_{i-j}).\]
%
%
By considering different isogenies $\pi: A_k \to B_k$, we can use
Corollary~\ref{cor@compaadd} to study the associativity of chain
additions:

\begin{corollary}
  \label{cor@fastmul}
  Let $x \in B_k[\ell]$ and $y \in B_k$. Choose any affine lifts $\tilde{x}$,
  $\tilde{y}$ and $\tilde{x+y}$ of respectively $x$, $y$ and $x+y$.
  \begin{enumerate}
    \item 
  Let $\tilde{nx+y}=\chainemultadd(n,\tilde{x+y},\tilde{x},\tilde{y})$ and
  $\tilde{nx}=\chainemult(n,\tilde{x})$.

  We have
  \begin{equation}
   \tilde{(n_1+n_2)x+y}=\chaineadd(\tilde{n_1 x+y},\tilde{n_2 x}, \tilde{
  (n_1-n_2) x +y})
    \label{eq@fastmul}
  \end{equation}

  In particular, we see that we can compute $\tilde{nx+y}$ in $O(\log(n))$
  addition chains by using a Montgomery ladder~\cite[alg. 9.5]{MR2162716}.

\item $-\tilde{nx+y}=
  \chaineadd(n,-(\tilde{x+y}),-\xtilde,-\ytilde)$
  \end{enumerate}
\end{corollary}
\begin{proof}
 We prove the two assertions.

  \begin{enumerate}
    \item 
  Let $K$ be a maximal isotropic group containing $x$, and let
  $A_k=B_k/K$. 
  Let
  $\pi: A_k \to B_k$ be the contragredient isogeny, and choose any theta structure on
  $(A_k, \pi^\ast \pol_0)$ compatible with $\pi$. 
There exist $i \in \Zl$ and $\lambda_i \in \overline{k}^*$
  such that $\xtilde= \lambda_i \pitilde_i(\zeroA)$. 
  If $\lambda_i=1$, then
  Corollary~\ref{cor@fastmul} is a consequence of Corollary~\ref{cor@compaadd}.
  But it is easy (see Lemma~\ref{lem@eqhomogen}) to see that~\eqref{eq@fastmul} is homogeneous in
  $\lambda_i$, hence the result.
\item 
  Once again, let $i \in \Zl$ be such that $\xtilde=\lambda_i \pitilde\left(
  (1,i,0).\zeroA \right)$, and let $\tilde{y}'$ be any point in
  $\pitilde^{-1}(\ytilde)$. By homogeneity we may suppose that
  $\lambda_i=1$. By Corollary~\ref{cor@compaadd} and
  Proposition~\ref{prop@compaadd}, we have
  $\tilde{n x+y}= \pitilde\left( (1,n.i,0).\ytilde' \right)$. 
  Now by Lemma~\ref{lem@oppose_point}, we have
  $- \tilde{n x+y}=\pitilde\left( - (1,n.i,0). \ytilde' \right) 
                  =\pitilde\left( (1,-n.i,0). -\ytilde' \right)
                  = \chaineadd(n,-(\tilde{x+y}),-\xtilde,-\ytilde)$.
  \end{enumerate}
\end{proof}

We make a last remark concerning Corollary~\ref{cor@compaadd} namely a
useful fact for studying the case $\ell$ not prime to $n$:
\begin{remark}
   Let $\xtilde \in \Atilde$, $i \in \Zln$ and let
   $\ytilde=\pitilde(\xtilde)$. Let $m \in \Z$ such that $\ell | m$.
   By Proposition~\ref{prop@compaadd} and Corollary~\ref{cor@compaadd}, we
   have
   \[ \pitilde \left( (1, m i, 0).\xtilde \right) =
   \chainemultadd(m, \pitilde_i(\xtilde), \Rtilde_i, \ytilde) \]
   But if $\ell | m$, then $m i \in \Zn \subset \Zln$, by
   Proposition~\ref{prop@compaaction} we have
   $\pitilde \left( (1, m i, 0).\xtilde \right) = (1, m i, 0). \ytilde$,
   and $(1, m i, 0). \ytilde$ can be computed with the
   formulas~\eqref{eq@actiontheta}.
   Hence
   \[ (1, m i, 0).\ytilde  = \chainemultadd(m, \pitilde_i(\xtilde), \Rtilde_i, \ytilde) \]
   \label{rem@lnotprime}
\end{remark}

For the purpose of Section~\ref{subsec@trueltorsion}, we have to study the
addition relations  between the points in $B_k[\ell]$ which does not
necessarily belong to $\KpolO$. From Section~\ref{subsec@evaluation}, we
see that we need to understand the link between the addition relations and
the isogeny $\ellisotilde$. More generally, for the rest of this section we
will study the relationship between the addition relations and a general
isogeny. First
we take a closer look at the action of $s_{\Kpolb}$ on the addition
relations. Let $\Inv$ be the automorphism of the Theta group from
Section~\ref{subsec@modcorr} that permutes $K_1$ and $K_2$. 
Since $s_{\Kpolb}=\Inv \circ s_{\Kpola} \circ \Inv$ we see that it suffices
to study the action
of $\Inv$ on the addition relations.

\begin{proposition}
  \label{prop@compa_theta_inversion}
  Suppose that $x,y,u,v,x',y',u',v' \in \Atilde$ satisfy the
  general Riemann equations~\eqref{eq_riemann_rel}. Then $\Inv.x$, $\Inv.y$,
  $\Inv.u$, $\Inv.v$, $\Inv.x'$, $\Inv.y'$, $\Inv.u'$, $\Inv.v'$ also satisfy~\eqref{eq_riemann_rel}. 
\end{proposition}
\begin{proof}
  If $x=(x_i)_{i \in \Zln}$ we recall (see~\eqref{eq@inversion_action})
  that 
  \[\Inv.x=( \sum_{j \in \Zln} e(i,j) x_j)_{i \in \Zln}\]
  where $e=e_{\pol}$ is the commutator pairing.

  By hypothesis, we have for $i,j,k,l \in \Zln$ such that $i+j+k+l=2m$:
\begin{multline}
\big(\sum_{t \in \Ztwo}  \vartheta_{i+t}(x) \vartheta_{j+t}(y)\big).\big(\sum_{t \in \Ztwo} \theta_{k+t}(u) \theta_{l+t}(v)\big)= \\
\big(\sum_{t \in \Ztwo} \vartheta_{i'+t}(x') \vartheta_{j'+t}(y')\big).\big(\sum_{t \in \Ztwo} 
\theta_{k'+t}(u') \theta_{l'+t}(v')\big). 
\end{multline}

Let $A_{\chi,x,y,i,j}= \big(\sum_{t \in \Ztwo} \chi(t) \vartheta_{i+t}(x) \vartheta_{j+t}(y)\big)$. We have if $I,J,K,L \in \Zln$ are such that $I+J+K+L=2M$:
\begin{align*}
  A_{\chi,\Inv.x,\Inv.y,I,J} &= \sum_{T \in \Ztwo} \chi(T) \big(\sum_{i \in \Zln}
  e(I+T,i) \theta_i(x)\big) \big(\sum_{j \in \Zln} e(J+T,j)
  \theta_j(x)\big) \\
  &= \sum_{T \in \Ztwo, i,j \in \Zln} \chi(T) e(T,i+j) e(I,i)e(J,j) \theta_i(x)
  \theta_j(y)
\end{align*}
\begin{multline*}
  A_{\chi,\Inv.x,\Inv.y,I,J} A_{\chi, \Inv.u, \Inv.v, K,L} = \\
  \sum_{\substack{T_1,T_2 \in \Ztwo \\  i,j,k,l \in \Zln}} \chi(T_1+T_2) e(T_1,i+j) e(T_2,k+l) e(I,i) e(J,j) e(K,k) e(L,l) \theta_i(x) \theta_j(y) \theta_k(u) \theta_l(v) \\
  = \sum_{i,j,k,l \in \Zln} e(I,i) e(J,j) e(K,k) e(L,l) \theta_i(x) \theta_j(y) \theta_k(u) \theta_l(v) \\
  \big( \sum_{T_1, T_2 \in \Ztwo} \chi(T_1+T_2) e(T_1,i+j) e(T_2,k+l) \big)
  \label{eq@@eq_interm_rel1}
\end{multline*}
But
\begin{equation*}
  \big( \sum_{T_1, T_2 \in \Ztwo} \chi(T_1+T_2) e(T_1,i+j) e(T_2,k+l)
  \big)=
\begin{dcases}
      4^g &  \text{if $e(\cdot,i+j)=e(\cdot, k+l)=\chi$} \\
      0 &  \text{otherwise} \\
\end{dcases}
\end{equation*}
and $e(\cdot,i+j)=e(\cdot,k+l)$ (as characters on $\Ztwo$) iff there exists $m \in \Zln$ such that $i+j+k+l=2m$. Now since $I+J+K+L=2M$ we have $e(I+J, \cdot)=e(K+L, \cdot)$ so we have:
\begin{multline*}
  \lambda \sum_{t_1,t_2 \in \Ztwo}
   e(I,i+t_1) e(J,j+t_1) e(K,k+t_2) e(L,l+t_2) \theta_{i+t_1}(x) \theta_{j+t_1}(y) \theta_{k+t_2}(u) \theta_{l+t_2}(v)= \\
  \lambda e(I,i) e(J,j) e(K,k) e(L,l) \sum_{t_1,t_2 \in \Ztwo}
   \theta_{i+t_1}(x) \theta_{j+t_1}(y) \theta_{k+t_2}(u) \theta_{l+t_2}(v)= \\
  \lambda e(I,i) e(J,j) e(K,k) e(L,l) \sum_{t_1,t_2 \in \Ztwo}
   \theta_{i'+t_1}(x') \theta_{j'+t_1}(y') \theta_{k'+t_2}(u') \theta_{l'+t_2}(v)= \\
  \lambda e(I',i') e(J',j') e(K',k') e(L',l') \sum_{t_1,t_2 \in \Ztwo}
   \theta_{i'+t_1}(x') \theta_{j'+t_1}(y') \theta_{k'+t_2}(u') \theta_{l'+t_2}(v)
\end{multline*}
where $\lambda=4^g$ if $i+j+k+l=2m$ and $\lambda=0$ otherwise.
By combining these relations we find that
\[ 
  A_{\chi,\Inv.x,\Inv.y,I,J} A_{\chi, \Inv.u, \Inv.v, K,L} = 
  A_{\chi,\Inv.x',\Inv.y',I',J'} A_{\chi, \Inv.u', \Inv.v', K,L} \] 
which concludes the proof. (We remark that we did not need to use the
general Riemann relations with characters on our proof. By considering
$\Inv \circ \Inv$, this shows that the general Riemann relations with
characters are induced by the general Riemann relations without
characters.)
\end{proof}

\begin{corollary}
  Let $\tilde{x}, \tilde{y}, \tilde{x-y} \in \Atilde$, and let $i,j \in
  \Zln$, $k,l \in \dZln$. Then we have:
  \[ (1,i+j,k+l).\chaineadd(\tilde{x},\tilde{y}, \tilde{x-y}) = 
      \chaineadd( (1,i,k).\tilde{x}, (1,j,l).\tilde{y}, (1,i-j,k-l).\tilde{x-y}) \] 
  \label{cor@compaadd2}
\end{corollary}
\begin{proof}
  By Propositions~\ref{prop@compaadd} and~\ref{prop@compa_theta_inversion}
  we have
  \begin{equation}
   s_2(k+l).\chaineadd(\tilde{x},\tilde{y}, \tilde{x-y}) = 
      \chaineadd( s_2(k).\tilde{x}, s_2(l).\tilde{y}, s_2(k-l).\tilde{x-y}) 
    \label{eq@compaadd2}
  \end{equation}
  Now since $(1,i,k)=s_1(i)s_2(k)$, we conclude by combining
  Equations~\eqref{eq@compaadd1} and~\eqref{eq@compaadd2}.
\end{proof}

\begin{corollary}
  Suppose that $\tilde{x_1}, \tilde{y_1}, \tilde{u_1}, \tilde{v_1}, \tilde{x_2},
  \tilde{y_2}, \tilde{u_2}, \tilde{v_2} \in \Atilde$ satisfy the Riemann
  relations~\eqref{eq_riemann_rel}. 

      If $\pi: (A, \pol, \ThetaA) \to (B, \pol_0, \ThetaB)$ is an isogeny
      such that $\ThetaB$ is $\pi$-compatible with $\ThetaA$, then
      $\pitilde(\tilde{x_1})$, $\pitilde(\tilde{y_1})$, $\pitilde(\tilde{u_1})$,
      $\pitilde(\tilde{v_1})$, $\pitilde(\tilde{x_2})$, $\pitilde(\tilde{y_2})$,
      $\pitilde(\tilde{u_2})$, $\pitilde(\tilde{v_2}) \in \Btilde$ also satisfy
      the general Riemann Relations.

      In particular we have \[ \pitilde(\chaineadd(\xtilde, \ytilde,
      \tilde{x-y}) = \chaineadd(\pitilde(\xtilde), \pitilde(\ytilde),
      \pitilde(\tilde{x-y}))  \]

\end{corollary}

\begin{proof}
  By Lemma~\ref{lem@compa_isogenies}, this is the case for compatible
  isogenies of type~$1$. By Proposition~\ref{prop@compa_theta_inversion}
  this is also the case for compatible isogenies of type~$2$, which
  concludes since every compatible isogeny is a composition of isogenies
  of type~$1$ or~$2$.
\end{proof}

\section{Application of the addition relations to isogenies}
\label{sec@isogenies}

In this Section we apply the results of Section~\ref{sec@addition} to the
computation of isogenies (see Section~\ref{subsec@isogeny}). More
precisely, we present an algorithm to compute the isogeny $\pidual: B_k \to
A_k$ from the knowledge of the modular point $\zeroA$. We give in
Section~\ref{sec@velu} algorithms to compute $\zeroA$ from the kernel of
$\pidual$. 

Since the embedding of $A_k$ that we consider is given by a theta structure
of level $\overln$, a point $\pidual(x)$ is given by $(\ell n)^g$
coordinates, which get impractical when $\ell$ is high. In order to
mitigate this problem, in Section~\ref{subsec@compression}, we give a point
compression algorithm such that the number of coordinates of a compressed
point does not depend on $\ell$.

We recall that we have chosen in Section~\ref{subsec@evaluation}
$\zeroA=(a_i)_{i \in \Zln}$ thus that $\pitilde(\zeroA)=\zeroB$, and that
we have defined for $i \in \Zln$ $\Rtilde_i=(a_{i+j})_{j \in \Zn}$.

\subsection{Point compression}
\label{subsec@compression}

Suppose that $\ell$ is prime to $n$. 
We know that $\xtilde \in \Atilde$ can be recovered from
$(\pitilde_i(\tilde{x}))_{i \in \Zl}$, by $(\tilde{x})_{ni+\ell
j}=(\tilde{\pi}_i(\tilde{x}))_j$. 
If $(d_1,\cdots,d_g)$ is a basis of $\Zl$, we can prove that
$\xtilde$
can be easily computed from just $(\pitilde_{d_i}(\xtilde))_{i \in [1..g]}$ and
$(\pitilde_{d_i+d_j}(\xtilde))_{i, j \in [1..g]})$.
\glsadd{glo@ei}
\glsadd{glo@di}
If $(e_1, \cdots, e_g)$ is the canonical basis of $\Zln$, in the
following, we take as a basis of $\Zl$ the $d_i= n e_i$ if $i \in  [1..g]$.


\begin{proposition}
  \[\pitilde_{i+j}(\xtilde)=\chaineadd(\pitilde_i(\xtilde),\Rtilde_j,\pitilde_{i-j}(\xtilde)). \]
\end{proposition}
\begin{proof}
  We apply Corollary~\ref{cor@compaadd} with $\ytilde=\zeroA$,
  $\tilde{x-y}=\xtilde$, so that we have $\chaineadd(\xtilde,\ytilde,\tilde{x-y)}=\xtilde$. We obtain:
  \[\pitilde_{i+j}(\xtilde)=\chaineadd(\pitilde_i(\xtilde),\pitilde_j(\zeroA),\pitilde_{i-j}(\xtilde)). \]
\end{proof}

\begin{definition}
  Let $S \sub G$ be a subset of a finite abelian group $G$ such that $0_G
  \in S$.
  We note $S'$ the inductive subset of $G$ defined by
  $S'= S \union \{ x+y | x \in S', y \in S', x-y \in S' \}$.
  We say that $S$ is a chain basis of $G$ if $S'=G$.
\end{definition}

\glsadd{glo@Sfrak}
\begin{example}
  \label{ex@chainegen}
  Let $G=\Zl$. Let $\{ e_1, \cdots, e_g\}$ be the canonical basis of $G$. There are two cases to consider to get a chain basis of $G$:
  \begin{itemize}
    \item If $\ell$ is odd, then one can take
      \[ S= \{ 0_G, e_i, e_i+e_j \}_{i,j \in [1..g], i<j }\]
    \item If $\ell$ is even, we use
      \[ S= \{ 0_G, e_{i_1}, e_{i_1}+e_{i_2}, \cdots, e_{i_1}+\cdots+e_{i_g} \}_{i_1,\cdots,i_g \in [1..g], i_1<\cdots<i_g }\]
  \end{itemize}
  In each case, the chain basis $S$ is minimal, we call it the canonical
  chain basis $\Sfrak(G)$ of $G$.
\end{example}

We recall that we have defined a section $\Scal \subset \Zln$ of
$\Zln \to \Zn$ in Example~\ref{ex@section-torsion}. To this set we
associate a canonical chain basis $\Sfrak \subset \Scal$ as follow: if
$\ell$ is prime to $n$, then $\Scal=\Zl \subset \Zln$, and we define
$\Sfrak=\Sfrak(\Zl)=\{ d_1, \cdots, d_g, d_1+d_g,\cdots, d_{g-1}+d_g \}$.
Otherwise we will take $\Sfrak=\Sfrak(\Zln)$.

\begin{theorem}[Point compression]
 Let $\xtilde \in \Atilde$. Then $\xtilde$ is uniquely determined by
 $\zeroA$ and $\{ \pitilde_i(\xtilde) \}_{i \in \Sfrak}$.

 $\zeroA$ is uniquely determined by $\{ \pitilde_i(\zeroA) \}_{i \in \Sfrak}=
 \{\Rtilde_i\}_{i \in \Sfrak}$.
 \label{th@compressionpoints}
\end{theorem}

\begin{proof}
  By Proposition~\ref{prop@compaadd}
  we have $\pitilde_{i+j}(\xtilde)= \chaineadd(\pitilde_i(\xtilde),
  \pitilde_j(\zeroA), \pitilde_{i-j}(\xtilde), \zeroB)$.
  So by induction, from $\{\pitilde_i(x)\}_{i \in \Sfrak}$ we can compute
  every $\{\pitilde_i(x)\}_{i \in \Sfrak'}$. Since $\Sfrak'=\Scal$ (or
  contains $\Scal$ if $n$ is not prime to $\ell$), 
  Corollary~\ref{cor@determination} shows that $\xtilde$ is
  entirely determined by $\{\pitilde_i(x)\}_{i \in \Sfrak}$ 
  and $\{\pitilde_i(\zeroA)\}_{i \in \Sfrak}$.

  In particular, $\zeroA$ is entirely determined by 
  $\{\pitilde_i(\zeroA)\}_{i \in \Sfrak}$.
  But $\pitilde_i(\zeroA)=\pitilde(\Ptilde_i)$. 
  by Proposition~\ref{prop@compaaction} which concludes.
\end{proof}

In the description of the algorithms, we suppose that $\ell$ is prime to
$n$, so that $\Scal=\Zl \subset \Zln$.

\begin{algorithm}[Point compression]
  \begin{description}
    \item[Input] $\xtilde=(\thetatilde_i(\xtilde))_{i \in \Zln} \in \Atilde$
    \item[Output] The compressed coordinates $(\pitilde_i(\xtilde))_{i \in \Sfrak}$.
    \item[Step 1] For each $i \in \Sfrak$, output
      $(\pitilde_i(\xtilde))=(\thetatilde_{n i +\ell j}(\xtilde))_{j \in \Zn}$
  \end{description}
\end{algorithm}

\begin{algorithm}[Point decompression]
  \begin{description}
    \item[Input] The compressed coordinates $\pitilde(\xtilde)_{i \in \Sfrak}$ of $\xtilde$.
    \item[Ouput] $\xtilde=(\thetatilde_i(\xtilde))_{i \in \Zln} \in \Atilde$.
    \item[Step 1] Set $\Scal':=\Sfrak$.
    \item[Step 2] While $\Scal' \ne \Scal$, choose $i,j \in \Scal'$ such
      that $i+j \in \Scal \setminus \Scal'$ and $i-j \in \Scal'$.\\
      Compute $\pitilde_{i+j}(\xtilde)=\chaineadd(\pitilde_i(\xtilde),
      \Rtilde_j, \pitilde_{i-j}(\xtilde))$.\\
      $\Scal':=\Scal' \bigcup \{i+j\}$.
    \item[Step 3] For all $i \in \Zln$, write $i=n i_0 +\ell j$ and output
      $\thetatilde_i(x)= \left( \pitilde_{i_0}(\xtilde) \right)_j$.
  \end{description}
\end{algorithm}

\begin{algocomplexity}\label{algo@compress}
  By using repeatedly the formula from Proposition~\ref{prop@compaadd}:
  \[\pitilde_{i+j}(\xtilde)= \chaineadd(\pitilde_i(\xtilde), \Rtilde_j,
  \pitilde_{i-j}(\xtilde), \zeroB)\] we can reconstitute every
  $\pitilde_i(\xtilde)$ for $i \in \Zl$ in Step~2 since $\Sfrak$ is a chain basis of
  $\Zl$. We can then trivially recover the coordinates of $\xtilde$ in
  Step~3 since
  they are just a permutation of the coordinates of the
  $\{ \pitilde_i(\xtilde), i \in \Zl \}$ (see Section~\ref{cor@determination}).

  To recover $\xtilde$, we need to do $\# \Scal - \#\Sfrak = O(\ell^g)$ chain
  additions. The compressed point $\{\pitilde_i(\xtilde)\}_{i \in \Sfrak}$
  is given by $\#\Sfrak \times n^g$ coordinates.

  If $\ell n=2 n_0$ and $n_0$ is odd we see that we can
  store a point in $\Atilde$ with $2^g\left( 1+g(g+1)/2 \right)$
  coordinates ($4^g$ if $n_0$ is even) rather than $(2 n_0)^g$. 
\end{algocomplexity}

\subsubsection{Addition chains with compressed coordinates}
\label{subsubsec@compressed_addition}
Let $\xtilde$, $\ytilde$ and $\tilde{x-y} \in \Atilde$. 
Suppose that we have the compressed coordinates
$\{\pitilde_i(\xtilde)\}_{i \in \Sfrak}$, $\{\pitilde_i(\ytilde)\}_{i \in \Sfrak}$, $\{\pitilde_i(\tilde{x-y})\}_{i \in \Sfrak}$.
Then if $i \in \Sfrak$ we have by Corollary~\ref{cor@compaadd}
\[ \pitilde_i(\tilde{x+y})=\chaineadd(\pitilde_i(\xtilde),
\pitilde_0(\ytilde), \pitilde_i(\tilde{x-y}), \]
hence we may recover the compressed coordinates of $\tilde{x+y}$.

We can compare this with an addition chain with the full coordinates (of
level $\ell n$). By the formulas from Theorem~\ref{th@addition},
since $2 | n$ and the formulas sum over points of $2$-torsion, we see that
we are doing $ \# \Scal$ addition chains in $B_k$ of level~$n$. This mean
that we are doing the same addition chains in $B_k$ when we use a chain
addition with compressed coordinates and then use the point decompression
algorithm. But if we just need the compressed coordinates, the chain
additions with compressed coordinates are much
faster since we need to do only $\# \Sfrak$ addition chains of level~$n$.
In particular, since we can compute the multiplication by $m$ with chain
additions, we see that the cost of a multiplication by $m$ is
$O(\# \Sfrak \log(m))$ addition chains of level~$n$.

Since we can take $n=2$, this mean that the additions formulas of level~$2$
allows us to compute addition chains of any level. In particular the speed
up for these formulas given by~\cite{gaudry2007fast} affects all levels.


\subsection{Computing the dual isogeny}
\label{subsec@isogeny}

We recall that we have the following diagram:
\begin{center}
\begin{tikzpicture}[scale=2.5,auto, 
  text height=1.5ex, text depth=0.25ex, 
  myarrow/.style={->,shorten >=1pt}]
  \node[anchor=west] (B) at (0.5,-0.87) {$y \in B_k$} ;
  \node[anchor=west]  (A2) at (1,0) {$z \in A_k$} ;
  \node[anchor=west]  (A1) at (0,0) {$x \in A_k$} ;
  \draw[myarrow] (B) to node[swap] {$\pidual$} (A2);
  \draw[myarrow] (A1) to node[swap] {$\pi$} (B);
  \draw[myarrow] (A1) to node {$[\ell]$} (A2);
\end{tikzpicture}
\end{center}


Let $\ytilde \in \pB^{-1}(y)$ and 
let $\xtilde \in \Atilde$ be such that $\pitilde(\tilde{x})=\ytilde$.
Let $i \in \Zl$. In this section, we describe
an algorithm to compute $\pitilde_i(\ell.\xtilde)$. By using this
algorithm for $i \in \{ d_1, \cdots, d_g, d_1+d_2, \cdots d_{g-1}+d_g \}$,
we can then recover $\pidual(y)=\pA(\ell.\xtilde)$ (see
Theorem~\ref{th@compressionpoints}, $\{d_i\}_{i \in [1..g]}$ is the basis
of $\Zl$ defined in Section~\ref{subsec@compression}). We know that
$\pi_i(x)=y+R_i$. For $i \in \Zl$, we choose a point $\pi_i^a(x) \in
p_A^{-1}(y+R_i)$ so that for each $i \in \Zl$ there exists $\lambda_i \in
\overline{k}^*$ such that $\pitilde_i(\xtilde)=\lambda_i \pi_i^a(x)$. 
If $\xtilde'$ is another point in $\pitilde^{-1}(y)$, then we have 
$\pitilde_i(\xtilde')=\lambda_i' \pi_i^a(x)$, with $\lambda_i'=\zeta \lambda_i$,
$\zeta$ a $\ell$-root of unity. As a consequence, it is possible to recover
$\lambda_i$ only up to an $\ell^{th}$-root of unity, but this
information is sufficient to
compute $\pitilde_i(\ell.\xtilde)$:

\begin{theorem}
  \label{th@isogeny}
  For all $i \in \Zl$, 
\[ \pitilde_i(\ell.\xtilde) =\lambda_i^{\ell} 
    \chainemultadd(\ell, \pi_i^a(x), \ytilde, \Rtilde_i)), \]
  where $\lambda_i^{\ell}$ is determined by:
\begin{equation*}
  \ytilde = \lambda_i^{\ell} \chainemultadd(\ell,
  \pi_i^a(x), \Rtilde_i, \ytilde).
\end{equation*}
\end{theorem}

\begin{proof}
By Proposition~\ref{prop@compaadd} and
Lemma~\ref{lem@eqhomogen} we have:
\[ \pitilde_i(\ell.\xtilde)=\chainemultadd(\ell, \pitilde_i(\xtilde), \pitilde(\xtilde),  \pitilde(\Ptilde_i))
=\lambda_i^{\ell} \chainemultadd(\ell, \pi_i^a(x), \ytilde, R_i). \]

Now we only need to find the $\lambda_i^{\ell}$ for $i \in \Zl$. 
But by Proposition~\ref{prop@compaadd} and an easy recursion, we have 
$\xtilde=s_{\Kpola}(i)^\ell.\xtilde$ so that by Corollary~\ref{cor@compaadd}
and Lemma~\ref{lem@eqhomogen}
\begin{equation*}
  \pitilde(\xtilde)=\chainemultadd(\ell, \pitilde_i(\xtilde),
  R_i, \ytilde) = \lambda_i^{\ell}. \chainemultadd(\ell,
  \pi_i^a(x), R_i, \ytilde)
\end{equation*}
\end{proof}

\begin{algorithm}[The image of a point by the isogeny]
  \label{algo@isogeny}
  \begin{description}
    \item[Input] $y \in B_k$.
    \item[Output] The compressed coordinates of $\pidual(y) \in A_k$.
    \item[Step 1] For each $i \in \Sfrak$ compute $y+R_i$ and choose an
      affine lift $y_i$ of $y+R_i$.
    \item[Step 2] For each $i \in \Sfrak$, compute $\mathtt{ylR}_i:=
  \chainemultadd(\ell, y_i, \Rtilde_i, y_0)$ and $\lambda_i$ such that 
  $y_0 = \lambda_i \mathtt{ylR}_i$.
    \item[Step 3] For each $i \in \Sfrak$, compute
      $\pitilde(\pidual(y_0))_i= \lambda_i
  \chainemultadd(\ell, y_i, \ytilde, \Rtilde_i))$.
  \end{description}
\end{algorithm}

\begin{algocomplexity}
  In Step~3 we compute $ \pitilde_i(\pidual(y)) =\lambda_i^\ell
  \chainemultadd(\ell, y_i, \ytilde, \Rtilde_i))$ where $\lambda_i^\ell$ is
  given in Step~2 by by $\ytilde = \lambda_i^\ell \chainemultadd(\ell, y_i,
  \Rtilde_i, \ytilde)$.

  We can easily recover $\pidual(y)$ from the $\pitilde_i(\pidual(y)), i
  \in \Zl$, but we note that it is faster to only compute the
  $\pitilde_i(\pidual(y))$ only for $i \in \Sfrak$ (with the notations of
  Example~\ref{ex@chainegen} in the preceding section), 
  and then do a point
  decompression (see Algorithm~\ref{algo@compress}). 
  This last step is of
  course unnecessary if the compressed coordinates of $\pidual(y)$ are
  sufficient.

  To compute $\pitilde_i(\ell.x)$, we need to do two multiplication chains
  of length $\ell$. We obtain the compressed coordinates of $\ell. x$ after
  $g (g+1)/2$ such operations. In total we can compute the compressed
  coordinates of a point in $O(\frac{1}{2} g(g+1) \log(\ell))$ additions in $B_k$
  (with $\frac{1}{2} g (g+1) n^g$ divisions in $k$)  and
  the full coordinates in $O(\ell^g)$ additions in $B_k$. We recover the
  equations of the isogeny by applying this algorithm to the generic point
  of $B_k$.
\end{algocomplexity}

\paragraph*{The case $(n,\ell)>1$}
In this case we have to use $\Sfrak=\{e_1, \cdots, e_g, e_1+e_2,\cdots \}$,
and in this case if $i \in \Sfrak$, $\Rtilde_{i}$ is a point of $\ell
n$-torsion.
But we have by Remark~\ref{rem@lnotprime}
\begin{equation*}
  (1,\ell i,0). \ytilde = \lambda_i^{\ell} \chainemultadd(\ell,
  \pi_i^a(x), \tilde{R}_i, \ytilde),
\end{equation*}
so that we can still recover $\lambda_i^{\ell}$.

\paragraph{The case $n=2$}
The only difficult part here is the ordinary additions $y+R_i$, since
the addition chains do not pose any problems with $n=2$. In particular, we
first choose one of the two points $\pm (x \pm R_{e_1})$, which requires a
square root. Now, since we have $\zeroA$ given by a theta structure of
degree $\ell n > 2$, we have the coordinates of $R_{e_1}+R_i$ on $B_k$.
This means that we can compute the compatible additions $x+R_i$ from
$x+R_{e_1}$ and $R_{e_1}+R_i$.

\subsection{Computation of the kernel of the isogeny}
\label{subsec@isogenykern}

We know that the kernel of the isogeny $\pidual: B \to A$ is the subgroup
$K$ generated by $\{R_{d_i}\}_{i \in [1..g]}$.
Let $\ytilde \in \Btilde[\ell]$, up to a projective factor, we may suppose that
$\chainemult(\ell, \ytilde)=\zeroB$. Then $y$ is in $K$ if and only if for
all $i \in \Zl$ we have $\pitilde_i(\pidual(\ytilde))=\Rtilde_i$. Let $\tilde{y+R_i}$ be any
affine point above $y+R_i$. Since $y$ and $R_i$ are points of 
$\ell$-torsion, for all $i \in \Zl$, there exist $\alpha_i, \beta_i
\in \overline{k}^*$ such that 
$\chainemultadd(\ell, \tilde{y+R_i}, \ytilde, \Rtilde_i)) = \alpha_i \Rtilde_i$
$\chainemultadd(\ell, \tilde{y+R_i}, \Rtilde_i, \ytilde)= \beta_i \ytilde$. By
Theorem~\ref{th@isogeny}, we know that
$\pitilde_i(\pidual(\ytilde))=\frac{\alpha_i}{\beta_i} \Rtilde_i$. 
In particular $y \in K$ if and only if $\frac{\alpha_i}{\beta_i}=1$ for all
$i \in \Zln$.

In fact, we will show in Section~\ref{sec@pairings} that
$\alpha_i/\beta_i=e'_\ell(y,R_i)$ where $e'_\ell$ is the extended commutator pairing
from Section~\ref{subsec@evaluation}.
We obtain that $y$ is in $K$ if and only if $e_W(y,R_i)=1$ for $i \in \{d_1,
\cdots, d_g\}$.

\section{The computation of a modular point}
\label{sec@velu}

In the Section~\ref{subsec@velu} we explain how to compute the theta null
point $\zeroA$ from the knowledge of the kernel of $\pidual$. This section
introduces the notion of a ``true'' point of $\ell$-torsion, which is an
affine lift of a point of $\ell$-torsion that satisfy
Equation~\eqref{eq@trueltorsion}. We study this notion in
Section~\ref{subsec@trueltorsion}, and we use these results 
in Section~\ref{subsec@ComputingModular} where we study the computation of
all (or just one) modular points.

\subsection{An analog of Vélu's formulas}
\label{subsec@velu}
We have seen in Section~\ref{sec@addition} how to use the addition formula
to compute the isogeny $\pidual: B_k \to A_k$. For this computation, we need to
know the theta null point $(a_i)_{i \in \Zln}$ corresponding to $A_k$. In
this section, we explain how to  recover the theta null point $(a_i)_{i \in \Zln}$, given the kernel $K=\{ T_i \}_{i \in
\Zl}$ of $\pidual$, by using only the addition relations. This gives an
analog to Vélu's formulas for higher genus. As in the course of the
algorithm we have to take $\ell^{th}$-root in $k$, we suppose that $k$ is algebraically closed.
(If $k=\F_q$, with $\ell | q-1$ so that we have the $\ell$-root of unity,
we only have to work over an extension of degree~$\ell$ of $k$).

Let $\{T_{d_1}, \cdots, T_{d_g} \}$ be a basis of $K$. Let $(a_i)_{i \in
\Zln}$ be the theta null point corresponding to any theta structure on
$A_k$
compatible with the theta structure on $B_k$. The compatible automorphisms of
the theta structure on $A_k$ allows us to recover all the theta null
point of the compatibles theta structures on $A_k$, via the actions:

\begin{gather}
  \{ \Rtilde_i \}_{i \in \Zl} \mapsto \{ \Rtilde_{\psi_1(i)} \}_{i \in
  \Zl},
  \label{action1}\\
  \{ \Rtilde_i \}_{i \in \Zl} \mapsto \{ e(\psi_2(i),i) \Rtilde_i
  \}_{i \in \Zl},
  \label{action2}
\end{gather}
where $\psi_1$ is an automorphism of $\Zl$ and $\psi_2$ is a symmetric
endomorphism of $\Zl$ (see \cite[Prop. 7]{0910.4668v1}).
The $\Rtilde_i$ were defined in Section~\ref{subsec@evaluation}, and we
recall they determine $\zeroA$ entirely. In fact the results of Section~\ref{subsec@compression} show that $\zeroA$ is completely determined by 
$\{ \Rtilde_{d_i}, \Rtilde_{d_i+d_j} \}_{i,j \in [1..g]}$
where $d_1, \cdots, d_g$ is a basis of $\Zl$.

Up to an action~\eqref{action1} we may suppose that $\zeroA$
is such that $\pitilde_{d_i}(\zeroA)=T_{d_i}$. 
Fix $i \in \Zl$. Let $\tilde{T_i}$ be any affine point above $T_i$, we have
$\Rtilde_i=\lambda_i \tilde{T_i}$. Write $\ell=2\ell'+1$, since $R_i$ is a
point of $\ell$-torsion, we have $(1,\ell'+1,0).\Rtilde_i=-
(1,\ell',0).\Rtilde_i$. By Proposition~\ref{prop@compaadd} and
Lemma~\ref{lem@eqhomogen}, we have
\begin{gather}
  \chainemult(\ell'+1,\Rtilde_i)=-\chainemult(\ell',\Rtilde_i) \notag, \\
  \lambda_i^{(\ell'+1)^2}\chainemult(\ell'+1,\tilde{T_i})=-\lambda_i^{\ell'^2}
  \chainemult(k,\tilde{T_i}), \notag \\
  \lambda_i^{\ell} \chainemult(\ell'+1,\tilde{T_i})=-
  \chainemult(\ell',\tilde{T_i}). \label{eq@trueltorsion1}
\end{gather}

Hence we may find $\lambda_i$ up to an $\ell^{th}$-root of unity. If we apply
this method for $i \in \{d_1, \cdots, d_g, d_1+d_2, \cdots, d_{g-1}+d_g
\}$, we find $\Rtilde_i$ up to an $\ell^{th}$-root of unity. But the
action~\eqref{action2} shows that every choice of $\Rtilde_i$ comes from a valid
theta null point $\zeroA$.

\begin{algorithm}[Vélu's like formula]
  \label{algo@velu}
  \begin{description}
    \item[Input] $T_{d_1}, \cdots T_{d_g}$ a basis of the kernel $K$ of $\pidual$.
    \item[Output] The compressed coordinates of $\zeroA$, the theta null point of level $\ell n$
  corresponding to $\pidual$. 
\item[Step 1] For $i,j \in [1..g]$ compute the points
  $T_{d_i}+T_{d_j}$. Let $\Sfrak=\{ d_1, \cdots, d_g, d_1+d_2, \cdots d_{g-1}+d_g \}$.
\item[Step 2] For each $i \in \Sfrak$
  choose any affine lift $T'_i$ of $T_i$, and compute
  $(\beta^i_j)_{j \in \Zn}:=\chainemult(\ell',T'_i)$, and
  $(\gamma^i_j)_{j \in \Zn}:=\chainemult(\ell'+1,T'_i)$.
\item[Step 3] 
  For each $i \in \Sfrak$  compute $\alpha_i$ such that
  $(\gamma^i_j)_{j \in \Zn} =\alpha_i (\beta^i_{-j})_{j \in \Zn}$.
\item[Step 4] 
  For each $i \in \Sfrak$, output
  $\Rtilde_i:= (\alpha_i)^{\frac{1}{\ell}} \cdot T'_i$.
  \end{description}
\end{algorithm}
\begin{algocomplexity}
In Step~4 we compute
$\Rtilde_i$ be any of the $\ell$ affine lift of $T_i$ such that: 
  $\chainemult(\ell'+1,\Rtilde_i)=- \chainemult(\ell',\Rtilde_i)$. 
  Then $\{ \Rtilde_i \}_{i \in  \Sfrak}$ give the compressed coordinates of
  $\zeroA$, we can then recover $\zeroA$ by doing a point decompression
  (see Algorithm~\ref{algo@compress}).

  To find $\Rtilde_i$, we need to do two chain multiplications of length
  $\ell/2$, and then take an $\ell^{th}$-root of unity. After $g (g+1)/2$
  such operations, we obtain the compressed coordinates of a $\zeroA$, and
  we may recover the full coordinates of the corresponding $\zeroA$ using
  the point decompression algorithm \ref{algo@compress}. We remark that we
  only need the compressed coordinates of $\zeroA$ to compute the
  compressed coordinates of $\pidual$. In total we need to compute
  $g(g+1)/2$ $\ell^{th}$-roots of unity and $O(\frac{1}{2} g (g+1)
  \log(\ell))$ additions in $B_k$ to recover the compressed coordinates of
  $\zeroA$. We can then recover the full coordinates of $\zeroA$ at the
  cost of $O(\ell^g)$ additions in $B_k$.
\end{algocomplexity}

\begin{remark}
  Each choice of the $g(g+1)/2$ $\ell^{th}$-roots of unity give a theta null
  point corresponding to the same Abelian variety $A_k=B_k/K$. However, each
  such point comes from a different theta structure on $A_k$, and hence give
  a different decomposition of the $\ell$-torsion $A[\ell]=K_1(\ell) \oplus
  K_2(\ell)$. Since $B_k=A_k/K_2(\ell)$, $K_2(\ell)=K$ is fixed so that each
  point gives a different $K_1(\ell)$. This mean that if $C_k=A_k/K_1(\ell)$
  we can recover different $\ell^2$-isogeny $B_k \to C_k$ from such
  choices (see
  Section~\ref{subsec@modcorr}).  By looking at the action~\eqref{action2},
  we see that there is a bijection between the $\ell^{g(g+1)/2}$ choices
  and the $\ell^2$ isogenies whose kernel $\mathfrak{K} \subset B_k$ is such
  that $\mathfrak{K}[\ell]=K$.
\end{remark}

\paragraph{The case $(n,\ell)>1$.}
In this case once again we have to recover $\Rtilde_{i}$ for $i \in \Sfrak=
\{ e_1, \cdots, e_g, e_1+e_2, \cdots, e_1+e_g \}$. Suppose that we have $\{
T_i \}_{i \in \Zl}$, $\ell^g$ points of $\ell n$-torsion such that
$\ell. T_i = (1, \ell i, 0).0_B$. Once again if $i \in \Sfrak$, we may
suppose that $\Rtilde_i=\lambda_i \tilde{T}_i$.

We have if $\ell=2\ell'+1$ is odd:
\begin{gather*}
  \lambda_i^{\ell} \chainemult(\ell'+1,\tilde{T_i})=- (1,\ell (n-1),0). \chainemult(\ell',\tilde{T_i})
\end{gather*}
so that once again we can find $\lambda_i^{\ell}$.

The kernel of $\pidual$ is then $K=\{ n T_i \}_{i \in \Zl}$. Even if $K$ is
isotropic, the $ \{T_i \}_{i \in \Zl}$ may not be, so some care must be
taken when we choose the $\{T_i \}_{i \in \Zl}$.

If $\ell=2 \ell'$ is even, we have:
\begin{gather*}
  \lambda_i^{2 \ell} \chainemult(\ell'+1,\tilde{T_i})=- (1,\ell (n-1),0). \chainemult(\ell'-1,\tilde{T_i}) 
\end{gather*}
so we can recover only $\lambda_i^{2 \ell}$. But every choice still
corresponds to a valid theta null point $(a_i)_{i \in \Zln}$, because when
$2 | \ell$, to the actions~\eqref{action1} and~\eqref{action2} we have to
add the action given by the change of the maximal symmetric level
structure~\cite[Proposition~7]{0910.4668v1}.

\paragraph{The case $n=2$}
Once again, the only difficulty rest in the standard additions. Using
standard additions, we may compute $R_{e_1}\pm R_{e_2}, \cdots,
R_{e_1} \pm R_{e_g}$, making a choice each time. Then we can compute
$R_{e_i}+R_{e_j}$ by doing an addition compatible with
$R_{e_1}+R_{e_i}$ and $R_{e_1}+R_{e_j}$.

\subsection{Theta group and $\ell$-torsion}
\label{subsec@trueltorsion}

Let $\tilde{x} \in \Btilde$ be such that $\pB(x)$ is a point of
$\ell$-torsion. We say that $x$ is a ``true'' point of $\ell$-torsion
if $\xtilde$ satisfy (see~\eqref{eq@trueltorsion1}):
\begin{equation}
  \chainemult(\ell'+1,\xtilde)=- \chainemult(\ell',\xtilde). 
  \label{eq@trueltorsion}
\end{equation}

\begin{remark}
  \label{rem@trueltorsion}
If $\xtilde$ is a ``true'' point of $\ell$-torsion, then
Lemma~\ref{lem@eqhomogen} shows it is also the case for $\lambda \xtilde$
for any $\lambda$ an $\ell^{th}$-root of unity. 
\end{remark}

We have seen in the preceding Section the importance of taking lifts that
are ``true'' points of $\ell$-torsion. The aim of this section is to use
the results of Section~\ref{subsec@theta_and_addition} to show that the
addition chain of ``true'' points of $\ell$-torsion is again a
``true''-point of $\ell$-torsion. We will use this in
Section~\ref{subsec@ComputingModular} to compute ``true'' affine lifts of
$B_k[\ell]$ by taking as few $\ell^{th}$-roots as possible.

We will use the affine lifts of points in $B_k[\ell]$ that we have
introduced in Section~\ref{subsec@evaluation} to study this notion.
Let $\bpol = [\ell]^* \ppol$ on $B_k$ and $\Thetall$ a theta structure for
$\bpol$ compatible with $\ThetaB$. 
Recall that we note $\Bell$ the affine cone of $(B_k, \bpol)$, and $\ellisotilde$ the
morphism $\Bell \to \Btilde$ induced by $\elliso$. Since
$\bpol \simeq \ppol^{\ell^2}$, the natural action of $G(\bpol)$ on
$H^0(\bpol)$ give via $\Thetall$ an action of $\Hll$ on $H^0(\bpol)$. 

\begin{lemma}
  Let $y \in B_k[\ell]$, $\ytilde \in \pB^{-1}(y)$ and $\xtilde \in
  \ellisotilde^{-1}(\ytilde)$. Then there exists  
  $(\alpha, n i,n j) \in k^{\ell} \times \Zstruct{\overll} \times
  \dZstruct{\overll}$ such that $\xtilde=(\alpha,n i,nj). 0_{\Bell}$.
  Moreover, 
  $\ytilde$ is a true point of $\ell$-torsion if and only if
  $\alpha=\lambda_{i,j} \mu$ where $\mu$ is an
  $\ell^{th}$-root of unity and $\lambda_{i,j}=e_c(i,-j)^{\ell' n
  (\ell-1)}$.

  (If $x' \in \Bell$, then $x' \in \ellisotilde^{-1}(y)$ if and only if
$x'=(1,\ell i', \ell j').x$ where $(i',j') \in \Zstruct{\overll} \times
\dZstruct{\overll}$), so the class of $\alpha$ in $k^{\ast}/k^{\ast\ell}$
does not depend on $\xtilde$ but only on $\ytilde$).
  \label{lem@carac_trueltorsion}
\end{lemma}

\begin{proof}
  Since $p_{\Bell}(\xtilde) \in B_k[\ell^2]$, there is an element $h \in
  \Hll$ such that $\xtilde=h.0_{\Bell}$, with $h=(\alpha, ni, nj)$.
  By Remark~\ref{rem@trueltorsion}, we only need to
  check that $(\lambda_{i,j},n i,n j).0_{\Bell}$ is a ``true'' point of $\ell$-torsion.
  Let $m \in \Z$, and let $\xtilde_m=\chainemult(m,\xtilde)$,  
   $\ytilde_m=\chainemult(m,\ytilde)$. By 
  Corollary~\ref{cor@compaadd2} we have 
  $\xtilde_m=(1,m \cdot i, m \cdot j).0_{\Bell}$, and by
  Corollary~\ref{cor@compaadd2}
  $\ytilde_m=\ellisotilde (1,m \cdot i, m \cdot j).0_{\Bell}$.
  So by Lemma~\ref{lem@oppose_point}
  $\ytilde_{\ell'}=\ellisotilde (1,\ell' \cdot i, \ell' \cdot j).0_{\Bell}
              =e_c(ni,\ell n (\ell-1) j) \ellisotilde (1,\ell' \cdot i+\ell n (\ell-1) , \ell' \cdot j+\ell n(\ell-1)).0_{\Bell}
=\lambda_{i,j}^{-\ell} \ellisotilde (1,-(\ell'+1) \cdot i , -(\ell'+1) \cdot j).0_{\Bell} 
=\lambda_{i,j}^{-\ell} \ellisotilde (-\xtilde_{\ell'+1})
= - \lambda_{i,j}^{-\ell} \ytilde_{\ell'+1}$.
\end{proof}

\begin{proposition}
  \label{prop@add_trueltorsion}
  Let $\tilde{y_1}, \tilde{y_2}, \tilde{y_1-y_2} \in \Btilde$ be points of ``true'' $\ell$-torsion. Then $\tilde{y_1+y_2}:=\chaineadd(\tilde{y_1},\tilde{y_2},\tilde{y_1-y_2})$ is a ``true'' point of $\ell$-torsion.
\end{proposition}

\begin{proof}
  Let $(\alpha_1, i_1, j_1) \in \Hll$,
      $(\alpha_2, i_2, j_2) \in \Hll$,
      $(\alpha_3, i_3, j_3) \in \Hll$,
  be such that 
  \[ \ellisotilde (\alpha_1,i_1,j_1). 0_{\Bell}   = \tilde{y_1},\quad 
  \ellisotilde (\alpha_2,i_2,j_2). 0_{\Bell}   = \tilde{y_2}, \quad  
  \ellisotilde (\alpha_3,i_3,j_3). 0_{\Bell}   = \tilde{y_1-y_2} \] 
  By the Remark at the end of Lemma~\ref{lem@carac_trueltorsion}, we may
  suppose that $i_3=i_1-i_2$, $j_3=j_1-j_2$. Since 
  $\tilde{y_1}, \tilde{y_2}$ and $\tilde{y_1-y_2}$ are ``true'' points of
  $\ell$-torsion, by Remark~\ref{rem@trueltorsion} and
  Lemma~\ref{lem@carac_trueltorsion} we may suppose that
  $\alpha_1=\lambda_{i_1,j_1}$, $\alpha_2=\lambda_{i_2,j_2}$ and $\alpha_3=\lambda_{i_1-i_2,j_1-j_2}$.

  By Corollary~\ref{cor@compaadd2} and Lemma~\ref{lem@eqhomogen}, we have
  \[\tilde{y_1+y_2}=\frac{\lambda_{i_1,j_1}^2 \lambda_{i_2,j_2}^2}{\lambda_{i_1-i_2, j_1-j_2}} (1,i_1+i_2,j_1+j_2).0_{\Bell} =
  (\lambda_{i_1+i_2,j_1+j_2},i_1+i_2,j_1+j_2). 0_{\Bell},\]
  so $\tilde{y_1+y_2}$ is
  indeed a ``true'' point of $\ell$-torsion by
  Lemma~\ref{lem@carac_trueltorsion}.
\end{proof}

\subsection{Improving the computation of a modular point}
\label{subsec@ComputingModular}

In \cite{0910.4668v1}, to compute the modular points $\zeroA$, the following
algorithm was used: write the Riemann relations of level $\ell n$ to
get a system in which we
plug the known coordinates $\zeroB$. We obtain a system $S$ of with a
finite number of solutions, that
can be solved using a Gr\"obner basis algorithm. But even for $g=2$ and
$\ell=3$, the system was too hard to be solved with a general Gr\"obner basis
algorithm, so we had to design a specific one.

In this section we explain how, using the ``Vélu's''-like formulas
of Section~\ref{subsec@velu}, it is possible to recover every modular
point $\zeroA$ solution of the system $S$ from the knowledge of the $\ell$-torsion of $B_k$. We then discuss different methods
to compute the $\ell$-torsion in $B_k$.

\begin{algorithm}[Computing all modular points]
  \label{algo@ltorsion} 
  \begin{description}
    \item[Input]
  $T_1, \cdots, T_{2g}$ a basis of the $\ell$-torsion of $B_k$.
    \item[Output] 
  All $\ell$-isogenies.
  \end{description}

  We only give an outline of the algorithm, since we give a detailed
  description in Example~\ref{ex@ltorsion}:
  Compute any affine ``true'' $\ell$-torsion lifts $\tilde{T}_1$, $\cdots$,
  $\tilde{T}_{2g}$, $\tilde{T_1+T_2}$, $\cdots$, $\tilde{T_{g-1}+T_g}$, and then
  use addition chains to compute affine lifts $\tilde{T}$ for every point $T \in  B_k[\ell]$. By Proposition~\ref{prop@add_trueltorsion} $\tilde{T}$ is a ``true'' point of $\ell$-torsion.

  For every isotropic subgroup $K \subset B_k[\ell]$, take the
  corresponding lifts and use them to reconstitute the corresponding theta
  null point $\zeroA$ (see Section~\ref{subsec@velu}).
\end{algorithm}

\begin{example}
  \label{ex@ltorsion}
  Suppose that $\{ T_1, \ldots, T_{2g} \} $ is a symplectic basis of
  $B_k[\ell]$. (A symplectic basis is easy to obtain from a basis of the
  $\ell$-torsion, we just need to compute the discrete logarithms of some
  of the pairings between the points, and we can use
  Algorithm~\ref{algo@compute_pairings} to compute these).

  Let $\Thetall$ be any theta structure of level $\ell^2 n$ on
  $B_k$ compatible with $\ThetaB$, and $\zeroB'$ be the corresponding
  theta null point (see Section~\ref{subsec@evaluation}). 
  We may suppose (see Section~\ref{subsec@velu}) that
  $\tilde{T_1}=\ellisotilde (1,(n,0,\cdots,0),0).\zeroB'$,
  $\tilde{T_2}=\ellisotilde (1,(0,n,\cdots,0),0).\zeroB'$,
  \ldots,
  $\tilde{T_{g+1}}=\ellisotilde (1,0,(n,0,\cdots,0)).\zeroB'$,
  $\tilde{T_{g+2}}=\ellisotilde (1,0,(0,n,\cdots,0)).\zeroB'$,
  \ldots,
  $\tilde{T_1+T_{g+2}}=\ellisotilde (1,(n,0,\cdots,0),(0,n,0,\cdots,0)).\zeroB'$,
  \ldots

  Then by Corollary~\ref{cor@compaadd2}, using Algorithm~\ref{algo@ltorsion}, we compute the following affine lifts of the $\ell$-torsion:
  \begin{equation}
    \{ \ellisotilde (1,in,jn).\zeroB': i,j \in \{0,1,\cdots,\ell-1\}^g
    \subset \Zstruct{\ell^2 n} \}.
    \label{eq@computedltorsion}
  \end{equation}

  Now if $K \subset B_k[\ell]$ is an isotropic group, in the reconstruction
  algorithm~\ref{algo@velu} we need to compute points of the form
  $\ellisotilde (1,in,jn).\zeroB'$ for $i,j \in \Zstruct{\ell^2 n}$.
  But we have 
  \begin{align*}
  \ellisotilde (1,in,jn).\zeroB' &= 
  \ellisotilde \zeta^{\ell \beta n \cdot (i - \ell \alpha) n} 
    (1, \ell \alpha n,  \ell \beta n).
  (1,(i-\ell \alpha)n,(j-\ell \beta) n). \zeroB' \\ 
  &= 
  \ellisotilde \zeta^{\ell \beta n \cdot (i - \ell \alpha) n} 
  (1,(i-\ell \alpha)n,(j-\ell \beta) n). \zeroB',
  \end{align*}
  where $\alpha,\beta \in \Zstruct{\ell^2 n}$, and $\zeta$ is a $(\ell^2
  n)^{th}$-root of unity. As a consequence, we can always go back to a point computed
  in~\eqref{eq@computedltorsion} up to an $\ell^{th}$-root of unity.

  We give a detailed example with $g=1$, $\ell=3$, $n=4$.
  Let $B_k$ be an elliptic curve, with a theta structure $\ThetaB$ of level $n$.
  Let $T_1$, $T_2$ be a basis of $B_k[\ell]$, and choose ``true'' affine lifts
  $\tilde{T_1}, \tilde{T_2}, \tilde{T_1+T_2}$. Let $\Thetall$
  be any theta
  structure of level $\ell^2 n$ compatible with $\ThetaB$, and $\zeroB'$ be
  the corresponding theta null point (see
  Section~\ref{subsec@theta_and_addition}). We take
  $\Thetall$  such that 
  $\tilde{T_1}=\ellisotilde (1,n,0).\zeroB'$,
  $\tilde{T_2}=\ellisotilde (1,0,n).\zeroB'$,
  and $\tilde{T_1+T_2}=\ellisotilde (1,n,n).\zeroB'$.

  We have seen in~\eqref{eq@computedltorsion} that in the Algorithm~\ref{algo@ltorsion} we
  compute the points:
  $\ellisotilde (1,in,jn).\zeroB'$ for $i,j \in {0,1,\cdots,\ell-1} \subset \Z/\ell^2
  n\Z$.

  Now let $T=\ellisotilde (1,n,2n).\zeroB'$, $K=<\pB(T)>$ is an isotropic subgroup of
  $B_k[\ell]$. Let $A_k=B_k/K$, choose a compatible theta structure $\ThetaA$
  on $A$, and let $\zeroA$ be the associated theta null point.
  
  As usual, we define $\Rtilde_i=\pitilde_i(\zeroA)$ if $i \in \Z/\ell Z \subset
  \Z/ \ell n \Z$, and we may suppose (Section~\ref{subsec@velu}) that 
  $\ThetaA$ is such that $R_1=T$. More explicitly, if $n=4$ we have
  $\zeroA=(a_0,a_1,a_2,a_3,a_4,a_5,a_6,a_7,a_8,a_9,a_{10}, a_{11})$, 
  $\pitilde(\left( x_i \right)_{i  \in \Z/12 \Z}) = (x_0,x_3,x_6,x_9)$ so
  that
  $\Rtilde_0=(a_0,a_3,a_6,a_9)=\zeroB$ (Remember that we always choose $\zeroA$
  such that $\pitilde(\zeroA)=\zeroB$), $\Rtilde_1=(a_4,a_7,a_{10},a_1)$ and
  $\Rtilde_2=(a_8,a_{11}, a_2, a_5)$. Now by Theorem~\ref{th@compressionpoints}
  we know that $\zeroA$ is entirely determined by $\Rtilde_1$ (and $\zeroB$), in
  fact we have: $\Rtilde_2=\chaineadd(R_1,R_1,\zeroB)$. By
  Corollary~\ref{cor@compaadd2}, we have 
  \[ \Rtilde_2=\ellisotilde (1,2n,4n).\zeroB'= \ellisotilde \zeta^{2n\cdot 3n}
  (1,0,3n).(1,2n,n).\zeroB'= \zeta^{2n \cdot 3n} \ellisotilde
  (1,2n,n).\zeroB', \]
  where $\zeta$ is a $(\ell^2 n)^{th}$-root of unity.

  This shows that in the reconstruction step, we have to multiply the point
  $\ellisotilde (1,2n,n).\zeroB'$ which we have already computed by the
  $\ell$-root of unity $\zeta^{2n \cdot \ell n}$.
\end{example}

\begin{complexity}
   To compute an affine lift $\tilde{T_i}$, we have to compute an
   $\ell^{th}$-root of unity (and do some addition chains but we can reuse the
   results for the next step). Once we have computed the $\ell
   (2\ell+1)$-root, we compute the whole (affine lifts of) $\ell$-torsion
   by using $O(\ell^{2g})$ addition chains. We can now compute the pairings 
   $e(T_i,T_j)$ with just one division since we have already computed the
   necessary addition chain (see Section~\ref{sec@pairings}). From these
   pairings we can compute a symplectic basis of $B_k[\ell]$. This
   requires
   to compute the discrete logarithm of the pairings and can be done in
   $O(\ell)$. Using this basis, we can enumerate every isotropic subgroup
   $K \subset B_k[\ell]$, and reconstruct the corresponding theta null
   point with $O(\ell^g)$ multiplications by an $\ell^{th}$-root of unity.
\end{complexity}

\paragraph{The case $(n,\ell>1)$}
In this case, the only difference is that we have to compute $B_k[\ell n]$
rather than $B_k[\ell]$, and when $T_i$ is a point of $\ell n$-torsion, we
compute an affine lift $\tilde{T_i}$ such that:
\begin{gather*}
  \chainemult(\ell'+1,\tilde{T_i})=- (1,\ell (n-1),0).
  \chainemult(\ell',\tilde{T_i}). \label{eq@trueltorsion2bis}
\end{gather*}

\paragraph{The case $n=2$:}
This works as in Section~\ref{subsec@velu}, once we have computed
the $\tilde{T_{e_1}}+\tilde{T_{e_i}}$, we have to take compatible additions
to compute the $\tilde{T_{e_i}} + \tilde{T_{e_j}}$.

\paragraph{Computing the points of $\ell$-torsion in $B_k$:}

The first method is to use the Riemann relations of level
$\delta=(n,n,n,\cdots, \ell n)$. Let $\phi: \Zn \to \Zdelta$ be the
canonical injection. Let $\mathcal{M}_{\delta}$ be the modular
space of theta null points of level~$\delta$, and $V_J$ the subvariety
defined by the ideal $J$ generated by $a_{\phi(i)}=b_i$ for $i \in \Zn$.

Then to every (non degenerate) point solution correspond an isogeny $\pi:
A_k \to B_k$. The kernel of the contragredient of $\pi$ is then of the form $\{0_{B_k}, T,
2T, \cdots, (\ell-1)T\}$ where $T$ is a primitive point of $\ell$-torsion.
We may recover $T$ as follows: if $i \in \Z/\ell \Z$, let $\pi_i=\pi \circ
s_1(0,0,\cdots,n i)$, then we have $\pi_i(\zeroA)=i \cdot T$.

By using this method for every point solution (even the degenerate ones),
we find all the points of $\ell$-torsion (the degenerate point solutions
giving points of $\ell$-torsion that are not primitive~\cite[Theorem 4]{0910.4668v1}).

When we look at the equations of $V_J$, we see that we can reformulate them
as follows. Let $R_1, \cdots, R_{\ell-1}$ be the $\ell-1$ projections of
the generic point of $\Btilde \times \Btilde \times \cdots \times
\Btilde$ and let $R_0=\zeroB$. Then the equations on $B_k$ comes from the
addition relations: for all $i,j \in \Zl$,
\[ R_{i+j}=\chaineadd(R_i,R_j,R_{i-j}). \]

We see that we can describe the system with less unknowns and equations by
looking only at equations of the form:
\[ R_{2i}=\chaineadd(R_i,R_i,R_{0}), \]
\[ R_{2i+1}=\chaineadd(R_{i+1},R_i,R_{0}). \]
This requires $n^g O(\log(\ell))$ variables, and each equations is of
degree~$4$. 

A second method is to work directly over the variety $B_k$. The Riemann
relations~(\ref{eq@addition}) allows us to compute the ideal of
$\ell$-division in $B_k$. Here we have $n^g$ unknown, but the equations are
of degree $\ell^{2g}$. Contrary to the first method, we recover projective
points of $\ell$-torsion, so we have to compute an $\ell^{th}$-root of unity to
compute a ``true'' affine lift of $\ell$-torsion. But this mean that the
degree of our system is $\ell^{2g}$ rather than $\ell^{2g+1}$, so a
generic Gr\"obner basis algorithm will finish faster. In genus~$2$, using a
Core~2 with 4~GB of RAM, this allows us to compute up to $\ell=13$.

In general we prefer to work over the Kummer surface (so with $n=2$), since
it cuts the degree of the system by two. In genus~$2$,
Gaudry and Schost~\cite{gaudryrecord}
have an algorithm to compute the $\ell$-torsion on the Kummer surface using
resultants rather than a general purpose Gr\"obner-based algorithm. The
points are given in Mumford coordinates, but we can use the results of
Wamelen~\cite{Wamelen98} to have them in theta coordinates. With this
algorithm, we can go up to $\ell=31$. This algorithm is in
$\tilde{O}(\ell^6)$ (where we use the notation $\tilde{O}$ to mean we
forget about the $\log$ factors). The computation of the ``true'' affine
points of $\ell$-torsion from Algorithm~\ref{algo@ltorsion} is in
$\tilde{O}(\ell^4)$, and each of the $O(\ell^3)$ isogeny requires
$O(\ell^2)$ multiplication by an $\ell^{th}$-root of unity. In total we see that
we can compute all $(\ell,\ell)$-isogenies in $\tilde{O}(\ell^6)$ in genus~$2$. 

Lastly, if we know the zeta function of $B_k$ (for instance if $B_k$ comes
from complex multiplication), we can recover a point of $\ell$-torsion by
taking a random point of $B_k$ and multiplying it by the required factor.

\paragraph{Isogenies graph:} 

~

Usually when we compute every isogenies, this is to build isogenies graph.
However, our Vélu's like algorithm from Section~\ref{subsec@velu} gives us
a theta null point $0_{A_k}$ of level $\ell n$ from a point of level $n$.
We can use the Modular correspondence from Section~\ref{subsec@modcorr} to
go back to a theta null point $0_{C_k}$ of level $n$, but the
corresponding isogeny $B_k \to C_k$ is a $\ell^2$ isogeny, so with our method
we can only draw $\ell^2$-isogenies graphs.

There is however one advantage of using the intermediate step $0_{A_k}$:
since it is a theta null point of level $\ell n$, we have all the
$\ell$-torsion on $A_k$. Let $\pi_2: A_k \to C_k$ be the corresponding
isogeny, $K_2:=\pi_2(A_k[\ell])$ gives us half the $\ell$-torsion of $C_k$.
(to get an explicit description of $K_2$, just apply $\Inv$ to the results
of Section~\ref{subsec@action}). 
Since $K_2$ is the kernel of the dual
isogeny $C_k \to A_k$, this allows us to build an isogeny graph of $\ell^2$
isogenies where the composition of two such isogenies give an
$\ell^4$-isogeny and not (for instance if $g=2$) a
$(1,\ell^2,\ell^2,\ell^4)$-isogeny (it suffices to consider the isotropic
subgroups of $C_k[\ell]$ that intersect $K_2$ trivially).
The knowledge of $K_2$ also allows us to speed up the
computation of $C_k[\ell]$: 
In the following section, we will give an algorithm to compute the extended
commutator pairing $e$ on $C_k[\ell]$. 
Let $(G_1, \cdots, G_g)$ be a basis of $K_2$, and consider the system of
degree $\ell^{g+1}$ given
by the ideal of $\ell$-torsion and the relations $e(G_i, \cdot)=1$ for $i
\in [2..g]$. Let $H_1$ be a point in this system different from $<G_1>$ (it
suffices to check that $e(G_1, H_1) \ne 1$. We can now construct the system
of degree $\ell^{g}$ given by the ideal of $\ell$-torsion and the
relations $e(G_i, \cdot)=1$ for $i \ne 2$ and $e(H_1, \cdot)=1$; and look
for a solution $H_2$ such that $e(G_2, H_2) \ne 1$. We see that we can
construct a basis $G_1, \cdots, G_g, H_1, \cdots, H_g of C_k[\ell]$ by
solving a system of degree $\ell^{g+1}$, then of degree $\ell^{g}$, \ldots,
then of degree $\ell^2$. This is faster than solving the ideal of
$\ell$-torsion which is a system of degree $\ell^{2g}$.


\section{Pairing computations}
\label{sec@pairings}
In this section, we explain how to use the addition chains introduced
in Section~\ref{subsec@pseudoadd} in order to
compute the commutator, Weil and Tate pairings on Abelian
varieties. We suppose here that $B_k$ is provided with a polarization
$\ppol$ which is the $n^{th}$ power of a principal polarization.
This allows us to make the link between the extended commutator pairing and
the Weil pairing in Section~\ref{subsec@commutator}. We then give an
algorithm to compute the extended commutator pairing in
Section~\ref{subsec@compute_commutator}.
\subsection{Weil pairing and commutator pairing}
\label{subsec@commutator}

We recall the definition of the extended commutator pairing from
Section~\ref{subsec@evaluation}: Let $\bpol = [\ell]^* \ppol$ on $B_k$. As
$\ppol$ is symmetric, we have that $\bpol \simeq \ppol^{\ell^2}$ and as a
consequence $K(\bpol)$, the kernel of $\bpol$ is isomorphic to
$K(\overline{\ell^2 n})$. The polarization $\bpol$ induces a commutator
pairing $e_{\bpol}$ (\cite{MR34:4269}) on $K(\bpol)$ and as $\bpol$
descends to $\ppol$ via the isogeny $[\ell]$, we know that $e_{\bpol}$ is
trivial on $B_k[\ell]$. For $x_1,x_2 \in B_k[\ell]$, let $x'_1, x'_2 \in
B_k[\ell^2]$ be such that $\ell.x'_i =x_i$ for $i=1,2$. The extended
commutator pairing is then
$e'_\ell(x_1,x_2)=e_{\bpol}(x'_1,x_2)=e_{\bpol}(x_1,x'_2)$, this is a well 
defined bilinear application $e'_\ell : B_k[\ell] \times B_k[\ell] \rightarrow
\overline{k}$. As $e_{\bpol}$ is a perfect pairing, for any $x'_1 \in
B_k[\ell^2]$ there exits $x'_2 \in B_k[\ell^2]$ such that $e_{\bpol}(x'_1,
x'_2)$ is a primitive $\ell^{2th}$ root of unity. As a consequence, for any
$x_1 \in B_k[\ell]$ there exists $x_2 \in B_k[\ell]$ such that $e'_\ell(x_1,
x_2)$ is a primitive $\ell^{th}$ root of unity and $e'_\ell$ is also a perfect
pairing. 

As the kernel of
$\ppol$ is $B_k[n]$, we have an isogeny $B_k \rightarrow \hat{B}_k$
with kernel $B_k[n]$ and by
composing this isogeny on the right side of $e'_\ell$, we obtain a
perfect pairing $e'_W : B_k[\ell] \times \hat{B}_k[\ell] \to
\mu_{\ell}$ where $\mu_\ell$ is the subgroup of $\ell^{th}$-roots of
unity of $\overline{k}$. 

We have
\glsadd{glo@eW}
\begin{proposition}
  The pairing $e'_W$ is the Weil pairing $e_W$.
\end{proposition}
\begin{proof}
 
  For $y \in \hat{B}_k[\ell]$, we denote by $\Lambda_y$ the degree-$0$
  line bundle on $B_k$ associated to $y$.  We first recall a possible
  definition of the Weil pairing $e_W$. Let $(x,y) \in B_k[\ell] \times
  \hat{B}_k[\ell]$. Let $\mathscr{O}_{B_k}$ be the structural sheaf of
  $B_k$, and as $y \in \hat{B}_k[\ell]$ there is an isomorphism
  $\psi'_y: [\ell]^* \Lambda_y \simeq \mathscr{O}_{B_k}$.  As a
  consequence, $\Lambda_y$ is obtained as the quotient of the trivial
  bundle $ B_k \times \Aff^1_k$ over $B_k$ by an action $g$ of
  $B_k[\ell]$ on $B_k \times \Aff^1_k$ given by $g_x(t, \alpha)=(t+x,
  \chi(x).\alpha)$ where $(t,\alpha) \in (B_k \times
  \Aff^1_k)(\overline{k})$, $x \in B_k[\ell]$ and $\chi$ is a
  character of $B_k[\ell]$. By definition~\cite{MR0282985}, we have
  $e_W(x,y)=\chi(x)$. 

  In order to give another formulation of this definition, we chose an
  isomorphism $\mathscr{O}_{B_k}(0) \simeq k$ from which we deduce via
  $\psi'_y$ (resp. $\tau^*_x \psi'_y$) an isomorphism $\psi_0:
  [\ell]^* \Lambda_y(0) \simeq k$ (resp. $\psi_1: \tau^*_x[\ell]^*
  \Lambda_y(0) \simeq k$). There exists a unique isomorphism $\psi_x :
  [\ell]^* \Lambda_y \rightarrow \tau^*_x [\ell]^* \Lambda_y$
  compatible on the $0$ fiber with $\psi_0$ and $\psi_1$, i.e. we have
  that $\psi_1 \circ \psi_x \circ \psi_0^{-1}$
  is the identity of $k$.
  Then, the following diagram commutes up to a multiplication by
  $e_W(x,y)$:

\begin{center}
\begin{tikzpicture}[auto, 
  text height=1.5ex, text depth=0.25ex, 
  myarrow/.style={->,shorten >=1pt}]
    \matrix (W) [row sep=1.5cm, column sep=1.5cm, matrix of math nodes]
    { [\ell]^*\Lambda_y &  \mathscr{O}_{B_k} \\
    \tau^*_x[\ell]^*\Lambda_y &  \tau^*_x \mathscr{O}_{B_k} \\} ;
    \draw[myarrow] (W-1-1) to node {$\psi'_y$} (W-1-2);
    \draw[myarrow] (W-2-1) to node {$\tau^*_x \psi'_y$} (W-2-2);
    \draw[myarrow] (W-1-1) to node {$\psi_x$} (W-2-1);
    \draw[double distance=3pt] (W-1-2) to node {$e_W(x,y)$} (W-2-2);
  \end{tikzpicture}\label{diag1}
\end{center}

The polarization $\ppol$ gives the natural isogeny
$\phi_{\ppol}$, defined on geometric points by \begin{align*}
  \phi_{\ppol}(\overline{k}): B_k(\overline{k}) & \to
  \hat{B}_k(\overline{k}) \\ y & \mapsto \Lambda_y = \ppol \otimes
  (\tau^*_y \ppol)^{-1}. \end{align*}
  As a
  consequence, for $y \in \hat{B}_k[\ell]$ there exists $y_0
  \in B_k(\overline{k})$ such that $\Lambda_{y}
  = \ppol\otimes (\tau^*_{y_0} \ppol)^{-1}$. Let $y' \in B_k[\ell^2]$ be
  such that $\ell.y'=y_0$.  As $[\ell]^*\ppol= \bpol$, we have $[\ell]^*
  \Lambda_{y}= [\ell]^*(\ppol\otimes (\tau^*_{y_0} \ppol)^{-1})=\bpol\otimes
  (\tau^*_{y'} \bpol)^{-1}$. We remark that the isomorphism $\psi'_y :
  [\ell]^* \Lambda_y = \bpol \otimes (\tau^*_{y'} \bpol)^{-1} \to
  \mathscr{O}_{B_k}$ gives by tensoring on the right by $\tau_{y'}^*
  \bpol$ an isomorphism $\psi_{y'} : \bpol \rightarrow \tau^*_{y'}
  \bpol$. Thus, the following diagram is commutative up to a
  multiplication by $e_W(x,y)$: 

\begin{center}
\begin{tikzpicture}[auto, 
  text height=1.5ex, text depth=0.25ex, 
  myarrow/.style={->,shorten >=1pt}]
    \matrix (W) [row sep=1.5cm, column sep=1.5cm, matrix of math nodes]
    { \bpol &  \tau^*_{y'} \bpol \\
    \tau^*_x \bpol &  \tau^*_{x+y'} \bpol \\} ;
    \draw[myarrow] (W-1-1) to node {$\psi_{y'}$} (W-1-2);
    \draw[myarrow] (W-2-1) to node {$\tau^*_x \psi_{y'}$} (W-2-2);
    \draw[myarrow] (W-1-1) to node {$\psi_x$} (W-2-1);
    \draw[myarrow] (W-1-2) to node {$\tau^*_{y'} \psi_x$} (W-2-2);
\end{tikzpicture}
\end{center}
But this is exactly the definition of $e'_W(x,y)$ thus we have
$e'_W(x,y)=e_W(x,y)$.

\end{proof}

\subsection{Commutator pairing and addition chains}
\label{subsec@compute_commutator}

In this paragraph, we explain how to compute the pairing $e'_\ell$ using
addition chains. From Section~\ref{subsec@commutator} this give a
different algorithm to compute the Weil pairing than the usual Miller
loop~\cite{journals/joc/Miller04}.
We chose a theta structure $\Thetall$ for $\bpol$
compatible with $\ThetaB$ (see Section~\ref{subsec@modcorr}) from which we deduce a decomposition $K(\bpol)=K_1(\bpol) \times
K_2(\bpol)$ of $K(\bpol)$ into isotropic subspaces for the commutator
pairing and a basis $(\theta_i)_{i \in \Zll}$ of $H^0(\bpol)$. We
recall that there is a natural action of $G(\bpol)$ on $H^0(\bpol)$
which can transported via $\Thetall$ to an action of $\Hll$ on
$H^0(\bpol)$.  Let $x,y \in B_k[\ell]$, and $x',y' \in B_k[\ell^2]$ be such
that $\ell.x'=y$ and $\ell.y'=y$. We put $x'=x'_1 + x'_2$ and $y'=y'_1
+ y'_2$ with $x'_i,y'_i
\in K_i(\bpol)$, for $i=1,2$.
Let $\kappa : G(\bpol) \to K(\bpol)$ be
the natural projection and let $(\alpha_1, \alpha_2), (\beta_1,
\beta_2) \in \Zll \times \dZll$ be such that
$\kappa(\Thetall(\alpha_i))=x'_i$ and $\kappa(\Thetall(\beta_i))=y'_i$. 
This mean that we have $x'=(1,\alpha_1,\alpha_2).0_{A_k}$ and 
$y'=(1,\beta_1,\beta_2).0_{A_k}$.

\begin{lemma}\label{lem@tech1}
Let $i\in \Zll$ and put
$$s(1)= \frac{( (1,\alpha_1,0).(1,\beta_1,0).\theta_i)(\tilde{0}_{B_k})}{( (1,\alpha_1,0)
.\theta_i)(\tilde{0}_{B_k})}. \frac{\theta_i(\tilde{0}_{B_k})}{( (1,\beta_1,0).\theta_i)(\tilde{0}_{B_k})}.$$
For all $k \in \N$, we have
\begin{equation}
s(k)=\frac{( (1,\alpha_1,0).(1,k.\beta_1,0).\theta_i)(\tilde{0}_{B_k})}{( (1,\alpha_1,0)
.\theta_i)(\tilde{0}_{B_k})}. \frac{\theta_i(\tilde{0}_{B_k})}{(
(1,k.\beta_1,0).\theta_i)(\tilde{0}_{B_k})}=s(1)^k.
\end{equation}
\end{lemma}
\begin{proof}
Consider the degree-$0$ line bundle
$\Lambda = \tau_{y'_1}^* \bpol \otimes \bpol^{-1}$. We remark that as
$y'_1 \in K(\bpol)$, $\Lambda$ is isomorphic to the trivial line
bundle on $B_k$. Let $K$ be the
subgroup of $K_1(\bpol)$ generated by $x'_1$ and let $C_k$ be the
quotient of $B_k$ by $K$. The line bundle $\Lambda$ descends to a line
bundle $\Lambda'$ over $C_k$. As $\Lambda'$ has degree $0$, it is the
quotient of $B_k \times \Aff^1_k$ by an action of the form
$g'_x(t,\alpha)=(t+x, \chi_0(x).\alpha)$, where $(t,\alpha) \in B_k
\times \Aff^1_x$, $x \in K$, and  $\chi_0$
is a character of $K$.

As $\theta_i \in H^0(\bpol)$, we remark that $f=((1,\beta_1,
0).\theta_i)/(\theta_i)$ is a section of $\Lambda$. Thus, we have
$s(k)=f(k. x'_1)/f(0_{B_k})=\chi_0(k)$ and $s(k)=s(1)^k$. 
\end{proof}
\begin{remark}
  We remark that in the preceding lemma, $\alpha_1$ and $\beta_1$ play
  the same role and as a consequence can be permuted.
\end{remark}
We keep the notation of the beginning of this paragraph to state the 
\begin{proposition}
  \label{prop@compute_pairings}
  We put:
  $$L = \frac{( (1,\ell.\alpha_1 + \beta_1, \ell.\alpha_2 +
  \beta_2)\theta_i)(\tilde{0}_{B_k})}{( (1, \beta_1,
  \beta_2)\theta_i)(\tilde{0}_{B_k})}.\frac{\theta_i(\tilde{0}_{B_k})}{((1, \ell.\alpha_1,
  \ell.\alpha_2)\theta_i)(\tilde{0}_{B_k})}.$$
  $$R = \frac{( (1,\alpha_1 + \ell.\beta_1, \alpha_2 +
  \ell.\beta_2)\theta_i)(\tilde{0}_{B_k})}{( (1, \alpha_1,
  \alpha_2)\theta_i)(\tilde{0}_{B_k})}.\frac{\theta_i(\tilde{0}_{B_k})}{((1, \ell.\beta_1,
  \ell.\beta_2)\theta_i)(\tilde{0}_{B_k})}.$$
We have :
\begin{equation}
  e'_\ell(x,y)=L^{-1}.R.
\end{equation}
\end{proposition}
\begin{proof}
  In order to clarify the notations we denote by $e_c$ the canonical
  pairing between $\Zll$ and $\dZll$.
First, we compute $L$. We have: 
\begin{eqnarray*}
 (1,\ell.\alpha_1 + \beta_1, \ell.\alpha_2 +
  \beta_2)\theta_i & = & e_c(\ell.\alpha_1 + \beta_1,
  -\ell.\alpha_2-\beta_2)(1,\ell.\alpha+\beta_1,0)(1,0,\ell.\alpha_2
  +\beta_2)\theta_i \\
  & = & e_c(\ell.\alpha_1 + \beta_1,
  -\ell.\alpha_2-\beta_2)(1,\ell.\alpha_1+\beta_1,0)\theta_i. 
\end{eqnarray*}
In the same way, we have:
\begin{eqnarray*}
  (1,\beta_1,\beta_2)\theta_i=e_c(\beta_1,
  -\beta_2)(1,\beta_1,0)\theta_i, \\
  (1,\ell.\alpha_1, \ell.\alpha_2)\theta_i=e_c(\ell.\alpha_1,
  -\ell.\alpha_2)(1,\ell.\alpha_1,0)\theta_i.
\end{eqnarray*}
Taking the product, we obtain that
\begin{equation*}
  L = e_c(\ell.\alpha_1, -\beta_2).L',
\end{equation*}
with
\begin{equation*}
  L'=\frac{(
  (1,\beta_1,0).(1,\ell.\alpha_1,0)\theta_i)(\tilde{0}_{B_k})}{(1,\beta_1,0)\theta_i(\tilde{0}_{B_k})}\frac{\theta_i(\tilde{0}_{B_k})}{
  (1,\ell.\alpha_1,0)\theta_i(\tilde{0}_{B_k})}.
\end{equation*}
In the same manner, we have:
\begin{equation*}
  R=e_c(\ell.\beta_1, -\alpha_2).R',
 \end{equation*}
 with
 \begin{equation*}
  R'=\frac{(
  (1,\alpha_1,0).(1,\ell.\beta_1,0)\theta_i)(\tilde{0}_{B_k})}{(1,\alpha_1,0)\theta_i(\tilde{0}_{B_k})}\frac{\theta_i(\tilde{0}_{B_k})}{
  (1,\ell.\beta_1,0)\theta_i(\tilde{0}_{B_k})}.
\
 \end{equation*}
 Using lemma \ref{lem@tech1} and the fact that $(1,\alpha_1,0)$
 commutes with $(1,\beta_1,0)$ we get that $L'=R'$. Therefore,
\begin{equation*}
  L^{-1}.R
  =e_c(\ell.\alpha_1,\beta_2).e_c(\ell.\alpha_2,\beta_1)=e'_\ell(x,y).
\end{equation*}
\end{proof}

The preceding proposition gives us an algorithm to compute the pairing:
\begin{algorithm}[Pairing computation]
  \label{algo@compute_pairings}
  \begin{description}
    \item[Input] $P, Q \in B_k[\ell]$
    \item[Output] $e'_\ell(P, Q)$
  \end{description}

  \gdef\Pcolor{\tilde{P}}
  \gdef\Qcolor{\tilde{Q}}
  \gdef\PQcolor{\tilde{P+Q}}
  Let $P,Q \in B_k[\ell]$, and choose any affine lift 
  $\tilde{P}$,
  $\tilde{Q}$ and $\tilde{P+Q}$,
  we can compute the following via
  addition chains: 
\begin{center}
\begin{tikzpicture}[mydiag]
    \matrix (W) [mymatrix0, row sep=0.3cm, column sep=0.3cm]
    { \zeroB \& \tilde{P} \& 2\tilde{P} \& \ldots \& \ell \tilde{P}=\lambda_P^0 \zeroB   \\
    \tilde{Q} \& \tilde{P+Q} \& 2\tilde{P}+\tilde{Q} \& \ldots \& \ell \tilde{P}+\tilde{Q}= \lambda_P^1 \tilde{Q} \\
    2 \tilde{Q} \& \tilde{P}+2\tilde{Q} \&  \&   \&   \\
      \ldots \& \ldots \&  \&   \&   \\
      \ell \tilde{Q}=\lambda_Q^0 \zeroB \& \tilde{P}+\ell \tilde{Q}=\lambda_Q^1 P \&  \&   \&   \\
     } ;
  \end{tikzpicture}
\end{center}
Namely we compute:
\begin{align*}
\ell \tilde{P}:=\chainemult(\ell,\Ptilde) &\quad
\ell \tilde{Q}:=\chainemult(\ell,\Qtilde) \\
\ell \tilde{P}+\tilde{Q}:=\chainemultadd(\ell,\tilde{P+Q},\Ptilde,\Qtilde)
&\quad
\tilde{P}+\ell\tilde{Q}:=\chainemultadd(\ell,\tilde{P+Q},\Qtilde,\Ptilde).
\end{align*}
      
Then we have:
\begin{equation}
  \label{eq@pairing_computation}
  e'_\ell(P,Q)=\frac{\lambda_P^1 \lambda_Q^0}{\lambda_Q^1 \lambda_P^0}
\end{equation}
\end{algorithm}

\begin{proof}
  Assume that 
  $\Pcolor$,
  $\Qcolor$ and $\PQcolor$ are such that 
  $\Pcolor= \ellisotilde (1,\alpha_1,\beta_1) \zeroB'$,
  $\Qcolor= \ellisotilde (1,\alpha_2,\beta_2) \zeroB'$,
  and $\PQcolor= \ellisotilde (1,\alpha_1+\alpha_2,\beta_1+\beta_2) \zeroB'$.
  Then by Corollary~\ref{cor@compaadd2}, we find that 
  $\lambda_P^0= \frac{\theta_i(0)}{((1, \ell.\alpha_1,
  \ell.\alpha_2)\theta_i)(0)} = 1$ and that $\lambda_P^1= 
  \frac{( (1,\ell.\alpha_1 + \beta_1, \ell.\alpha_2 +
  \beta_2)\theta_i)(0)}{( (1, \beta_1, \beta_2)\theta_i)(0)}$, so that by
  Proposition~\ref{prop@compute_pairings}, we have:
  \[ e'_\ell(P,Q)=\frac{\lambda_P^1 \lambda_Q^0}{\lambda_Q^1 \lambda_P^0}\]

  Now by Lemma~\ref{lem@eqhomogen}, it is easy to see
  that~\eqref{eq@pairing_computation} is homogeneous and does not depend on
  the affine lifts $\Pcolor$, $\Qcolor$ and $\PQcolor$, which
  concludes the proof.
\end{proof}

\begin{complexity}
    By using a Montgomery ladder, we see that 
    we can compute $e'_\ell(P,Q)$ with four
    fast addition chains of length~$\ell$, hence we need
    $O(\log(\ell))$ additions. It should be noted that we can reuse a lot
    of computation between the addition chains $P, 2P, 4P,\ldots$ and
    $P+Q, 2P+Q, 4P+Q, \ldots$ since we always add the
    same point at the same time between the two chains. 
\end{complexity}
  
\paragraph{The case $n=2$}
Let $\pm P, \pm Q \in K_B$, then we have $e'_\ell(\pm P, \pm Q)=\{ e'_\ell(P,Q),
e'_\ell(P,Q)^{-1}\}$. Thus the pairing on the Kummer variety is a bilinear
pairing $K_B \times K_B \to k^{\ast, \pm}$ where 
$k^{\ast,\pm}=k^{\ast}/\{ x=1/x\}$. We represent a class $\overline{x} \in
k^{\ast, \pm}$ by $x+1/x \in k$, and we define the symmetric pairing
$e'_{s}(\pm P,\pm Q)=e'_\ell(P,Q)+e'_\ell(P,-Q)$.
We can use the addition relations to compute $P \pm Q$ and then use
Algorithm~\ref{algo@compute_pairings} to compute $e'_\ell(P,Q)$, $e'_\ell(P,-Q)$,
but since $e'_{s}$ is symmetric there should be a method to compute it
without solving the degree-$2$ system given by the addition relations. This
will be the object of a future article.


\section{Conclusion}

We have described an algorithm that give a modular point from an isotropic
kernel, and another one that can compute the isogeny associated to a
modular point. By combining these two algorithms, we can compute any
isogeny between abelian varieties. However, the level of the modular space
that we use depend on the degree of the isogeny. That means that a point in
this modular space $\Mln$ corresponds to an isogeny and to the choice of a
symplectic basis of $B_k[\ell]$. So in our case it is easier to compute
directly the points of $\ell$-torsion in $B_k$ than to compute directly the
modular points in $\Mln$ since the degree of this latter system is higher.
These remarks mean that we cannot use our algorithm to
speed up Schoof point counting algorithm, like in the genus-one case (see
for instance~\cite{MR1486831}). A solution would be to have an efficient
characterization of $\Mln$ modulo the action by the symplectic group of
order~$\ell$.

Still, we can go back to a modular point of level~$n$ by using the
modular correspondence introduced in~\cite{0910.4668v1}. This mean that we
can compute isogeny graphs if we restrict to $\ell^2$-isogenies. We have
also introduced a point compression algorithm, that allows to drastically
reduce the number of coordinates of a projective embedding of level $4
\ell$. This new representation can be useful when one has to work with such
a projective embedding, rather than the usual one of level~$4$ (for
instance if one need a quick access to the translation by a point of
$\ell$-torsion). 


\nocite{MR85h:14026}

\bibliographystyle{alpha}
\bibliography{common}

\end{document}